%% file: anabfam.tex
\documentclass[11pt, a4paper]{amsart}
\input{styles/style}
\input{styles/macros}
\title[Anabelian Families]{Anabelian Geometry in Families}

\author{Tim Holzschuh}
\address{Tim Holzschuh, Institut des Hautes \'Etudes Scientifiques (IHES)\\ \newline \indent
Le Bois Marie\\
35 route de Chartres\\
91440 Bures-sur-Yvette\\
France}
\email{holzschuh@ihes.fr}

\author{Alexander Schmidt}
\address{Alexander Schmidt, Institut f\"{u}r Mathematik, Universit\"at Heidelberg \\ \newline \indent
Im Neuenheimer Feld 205, 69120 Heidelberg, Germany}
\email{schmidt@mathi.uni-heidelberg.de}

\author{Jakob Stix}
\address{Jakob Stix, Institut f\"{u}r Mathematik, Goethe-Universit\"{a}t Frankfurt \\ \newline \indent
Robert-Mayer-Stra{\ss}e~{6--8},
60325 Frankfurt am Main, Germany}
\email{stix@math.uni-frankfurt.de}

\date{\today}

\begin{document}

\hrule width\hsize

\vskip 0.8cm

\begin{abstract}
Anabelian geometry as an attempt to describe geometry in terms of étale topological data has addressed so far mainly categories of varieties over a field.
In this paper we work over a normal base scheme $S$ of finite type over a sub-$p$-adic field  and show that families of hyperbolic curves over $S$ are anabelian among smooth $S$-schemes with respect to dominant morphisms.
\end{abstract}

\maketitle

\setcounter{tocdepth}{1}
{\scriptsize \tableofcontents}

%%%%%%%%%%%%%%%%%%%%%%%%%%%%%%%%%%%%%%%%%%%%%%%%%%%%%%%
\section{Introduction}
\label{sec:introduction}
\input{content/introduction.tex}

%%%%%%%%%%%%%%%%%%%%%%%%%%%%%%%%%%%%%%%%%%%%%%%%%%%%%%%%
\section{Profinite anima and the étale homotopy type}
\label{sec:profinite-anima}
\input{content/profinite-anima.tex}

%%%%%%%%%%%%%%%%%%%%%%%%%%%%%%%%%%%%%%%%%%%%%%%%%%%%%%%%
\section{\texorpdfstring{Schemes of type  $\K(\pi,1)$}{Schemes of type K(π,1)}} 
\label{sec:BG}  
\input{content/K-pi-1-schemes.tex}

%%%%%%%%%%%%%%%%%%%%%%%%%%%%%%%%%%%%%%%%%%%%%%%%%%%%%%
\section{Anabelian geometry of hyperbolic curves revisited}
\label{sec:mochizuki}
\input{content/mochizuki-anima-version.tex}

%%%%%%%%%%%%%%%%%%%%%%%%%%%%%%%%%%%%%%%%%%%%%%%%%%%%%%%
\section{Quasifibrations, base change and homotopy pullbacks}
\label{sec:basechange}
\input{content/basechange.tex}

%%%%%%%%%%%%%%%%%%%%%%%%%%%%%%%%%%%%%%%%%%%%%%%%%%%%%%%
\section{Spreading out homotopies}
\label{sec:spreading-out-homotopies}
\input{content/spreading-out-homotopies.tex}

%%%%%%%%%%%%%%%%%%%%%%%%%%%%%%%%%%%%%%%%%%%%%%%%%%%%%%%
\section{Anabelian criterion for extending maps to curves}
\label{sec:extending-curves}
\input{content/anabelian-criterion-for-extending-curves.tex}

%%%%%%%%%%%%%%%%%%%%%%%%%%%%%%%%%%%%%%%%%%%%%%%%%%%%%%%
\section{Rigidity in profinite étale homotopy of hyperbolic curves}
\label{sec:finalize-proof}
\input{content/finalize-proof.tex}

%%%%%%%%%%%%%%%%%%%%%%%%%%%%%%%%%%%%%%%%%%%%%%%%%%%%%%%
\appendix
\input{content/appendix.tex}

%%%%%%%%%%%%%%%%%%%%%%%%%%%%%%%%%%%%%%%%%%%%%%%%%%%%%%%
\printbibliography
%%%%%%%%%%%%%%%%%%%%%%%%%%%%%%%%%%%%%%%%%%%%%%%%%%%%%%%

\end{document}

%%% Local Variables:
%%% mode: LaTeX
%%% TeX-master: t
%%% End:

%% file: styles/style.tex
\usepackage[T1]{fontenc}
\usepackage{stix2}

\usepackage{geometry}
\usepackage{microtype} % Subtle kerning, spacing improvements

\DeclareRobustCommand{\SkipTocEntry}[5]{}

\usepackage{amsthm}
\usepackage{dsfont}
\usepackage{mathtools} % Fixes and extends amsmath (e.g., \coloneqq, \prescript)

\usepackage{tikz-cd}
\tikzcdset{arrow style=tikz, diagrams={>=stealth}} % Clean, modern arrows

\usetikzlibrary{decorations.markings,decorations.pathmorphing}

\makeatletter
\tikzcdset{
  open/.code     = {\tikzcdset{hook, circled};},
  closed/.code   = {\tikzcdset{hook, slashed};},
  open'/.code    = {\tikzcdset{hook', circled};},
  closed'/.code  = {\tikzcdset{hook', slashed};},
 circled/.code  = {\tikzcdset{markwith = {\draw (0,0) circle (.375ex);}};},
 slashed/.code  = {\tikzcdset{markwith = {\draw[-] (-.4ex,-.4ex) -- (.4ex,.4ex);}};},
  markwith/.code ={
    \pgfutil@ifundefined%
    {tikz@library@decorations.markings@loaded}%
    {\pgfutil@packageerror{tikz-cd}{You need to say %
      \string\usetikzlibrary{decorations.markings} to use arrows with markings}{}}{}%
    \pgfkeysalso{/tikz/postaction = {
      /tikz/decorate,
      /tikz/decoration={markings, mark = at position 0.5 with {#1}}}
    }
  },
}
\makeatother
\usepackage{graphicx}
\usepackage{rotating} % for rotating e.g. isomorphism symbols in comm. diagrams
\usepackage[pdfpagelabels]{hyperref}
\hypersetup{
  pdftitle={Anabelian Geometry in Families},
  pdfauthor={Tim Holzschuh,  Alexander Schmidt, and Jakob Stix},
  pdfsubject={},
  pdfkeywords={},
  colorlinks=true,    % false: boxed links; true: colored links
  linkcolor=blue,     % color of internal links
  citecolor=blue,     % color of links to bibliography
  filecolor=blue,      % color of file links
  urlcolor=blue,       % color of external links
  breaklinks=true,
  bookmarksopen=true,
  bookmarksnumbered=true,
  pdfpagemode=UseOutlines,
  plainpages=false,
  unicode=true
  }
  
\numberwithin{equation}{subsection}

\usepackage{color}
\usepackage{soul}
\setuldepth{Berlin}

\usepackage{csquotes}

\usepackage[nameinlink, capitalise]{cleveref} % Must be loaded *after* amsart, amsthm, etc.

\creflabelformat{theorem}{#2\textup{#1}#3} % "Theorem 3.1"
\creflabelformat{lemma}{#2\textup{#1}#3}
\creflabelformat{equation}{#2(#1)#3} % "(3.1)"

\crefformat{section}{\S#2#1#3}
\crefformat{subsection}{\S#2#1#3}
\crefformat{subsubsection}{\S#2#1#3}

\crefname{thmABC}{Theorem}{Theorems}

\crefname{section}{\S\!\!}{\S\S\!\!}
\Crefname{section}{Section}{Sections}
\crefname{subsection}{\S\!\!}{\S\S\!\!}
\Crefname{subsection}{Subsection}{Subsections}

\crefname{conjecture}{Conjecture}{Conjectures}
\Crefname{conjecture}{Conjecture}{Conjectures}
\crefname{corollary}{Corollary}{Corollaries}
\Crefname{corollary}{Corollary}{Corollaries}
\crefname{definition}{Definition}{Definitions}
\Crefname{definition}{Definition}{Definitions}
\crefname{example}{Example}{Examples}
\Crefname{example}{Example}{Examples}
\crefname{lemma}{Lemma}{Lemmata}
\Crefname{lemma}{Lemma}{Lemmata}
\crefname{notation}{Notation}{Notations}
\Crefname{notation}{Notation}{Notations}
\crefname{proposition}{Proposition}{Propositions}
\Crefname{proposition}{Proposition}{Propositions}
\crefname{recollection}{Recollection}{Recollections}
\Crefname{recollection}{Recollection}{Recollections}
\crefname{remark}{Remark}{Remarks}
\Crefname{remark}{Remark}{Remarks}
\crefname{warning}{Warning}{Warnings}
\Crefname{warning}{Warning}{Warnings}

\AddToHook{env/lemma/begin}{\crefalias{theorem}{lemma}}
\AddToHook{env/remark/begin}{\crefalias{theorem}{remark}}
\AddToHook{env/corollary/begin}{\crefalias{theorem}{corollary}}
\AddToHook{env/definition/begin}{\crefalias{theorem}{definition}}
\AddToHook{env/proposition/begin}{\crefalias{theorem}{proposition}}
\AddToHook{env/example/begin}{\crefalias{theorem}{example}}
\AddToHook{env/recollection/begin}{\crefalias{theorem}{recollection}}
\AddToHook{env/notation/begin}{\crefalias{theorem}{notation}}
\AddToHook{env/warning/begin}{\crefalias{theorem}{warning}}
\AddToHook{env/construction/begin}{\crefalias{theorem}{construction}}

\newcommand{\lcref}[1]{\labelcref{#1}}

\newtheorem{theorem}{Theorem}[section]
\newtheorem{lemma}[theorem]{Lemma}
\newtheorem{proposition}[theorem]{Proposition}
\newtheorem{corollary}[theorem]{Corollary}

\newtheorem{thmABC}{Theorem}

\theoremstyle{definition}
\newtheorem{definition}[theorem]{Definition}

\newtheorem{remark}[theorem]{Remark}
\newtheorem{remarks}[theorem]{Remarks}
\newtheorem{recollection}[theorem]{Recollection}

\newtheorem{construction}[theorem]{Construction}

\newtheorem*{definition*}{Definition}
\newtheorem*{remark*}{Remark}

\usepackage[shortlabels]{enumitem}
\setlist{leftmargin=.8cm}

\makeatletter
\def\cref@thmoptarg[#1]#2#3#4{%
    \ifhmode\unskip\unskip\par\fi%
    \normalfont%
    \trivlist%
    \let\thmheadnl\relax%
    \let\thm@swap\@gobble%
    \thm@notefont{\fontseries\mddefault\upshape}%
    \thm@headpunct{.}% add period after heading
    \thm@headsep 5\p@ plus\p@ minus\p@\relax%
    \thm@space@setup%
    #2% style overrides
    \@topsep \thm@preskip               % used by thm head
    \@topsepadd \thm@postskip           % used by \@endparenv
    \def\@tempa{#3}\ifx\@empty\@tempa%
      \def\@tempa{\@oparg{\@begintheorem{#4}{}}[]}%
    \else%
      \refstepcounter[#1]{#3}%  <<< cleveref modification
      \@namedef{cref@#3@alias}{#1}% added
      \def\@tempa{\@oparg{\@begintheorem{#4}{\csname the#3\endcsname}}[]}%
    \fi%
    \@tempa}%
\makeatother

\usepackage[style=alphabetic, backend=biber, natbib=true,
    url=false,
    doi=false, %true,
    isbn=false,
    eprint=false, 
    maxalphanames=3, minalphanames=3,
    maxbibnames=99]{biblatex}
\DeclareFieldFormat[article, inbook, incollection, inproceedings, misc, thesis, unpublished]{title}{\textit{#1}}
\DeclareFieldFormat{journaltitle}{{\rmfamily #1}}
\DeclareFieldFormat{booktitle}{{\rmfamily #1}}
\DeclareFieldFormat{pages}{#1}
\DeclareFieldFormat{url}{%
  \url{#1}%
}
\renewrobustcmd*{\bibinitdelim}{\addnbthinspace{}}
\renewrobustcmd*{\bibnamedelima}{\addnbthinspace{}}
\renewrobustcmd*{\bibnamedelimd}{\addnbthinspace{}}

\DeclareBibliographyDriver{article}{%
  \usebibmacro{author/editor}%
  \newunit
  \usebibmacro{title}%
  \newunit
  \usebibmacro{journal}%
  \setunit*{\addspace}%
  \printtext{\textbf{\printfield{volume}}}%
  \setunit*{\addspace}%
  \printtext[parens]{\printfield{year}}%
  \setunit{\addcomma\space}%
  \iffieldundef{number}{}{\printtext{no.~\printfield{number}}}%
  \setunit{\addcomma\addspace}%
  \printfield{pages}%
  \iffieldundef{pubstate}{}{\addcomma\addspace\printfield{pubstate}%
  	 \iffieldundef{doi}{}{\addcomma\addspace\printfield{doi}}%
	 }%
  \usebibmacro{finentry}}

\newlist{deflist}{enumerate}{2}
\setlist[deflist, 1]{
  itemsep    = 0.2cm,
  label      = {\upshape (\alph*)},
  ref        = {(\alph*)},
  leftmargin = *
}
\setlist[deflist, 2]{
  itemsep    = 0.2cm,
  label      = {\upshape (\arabic*)},
  ref        = {(\arabic*)},
  leftmargin = *
}

\newlist{thmlist}{enumerate}{2}
\setlist[thmlist, 1]{
  itemsep    = 0.2cm,
  label      = {\upshape (\arabic*)},
  ref        = {(\arabic*)},
  leftmargin = *
}
\setlist[thmlist, 2]{
  itemsep    = 0.2cm,
  label      = {\upshape (\alph*)},
  ref        = {(\alph*)},
  leftmargin = *
}

\newcommand{\SAGsubseclink}[1]{\href{https://www.math.ias.edu/~lurie/papers/SAG-rootfile.pdf\#subsection.#1}{#1}}

\newcommand{\HAthmlink}[1]{\href{https://www.math.ias.edu/~lurie/papers/HA.pdf\#theorem.#1}{#1}}

\newcommand{\stackstag}[1]{\href{https://stacks.math.columbia.edu/tag/#1}{Tag #1}}

\newcommand{\kerodontag}[1]{\href{https://kerodon.net/tag/#1}{Tag #1}}

\newcommand{\HTTsubsec}[1]{\href{https://www.math.ias.edu/~lurie/papers/HTT.pdf\#subsection.#1}{\S #1}}

\newcommand{\SAGsubsec}[1]{\href{https://www.math.ias.edu/~lurie/papers/SAG-rootfile.pdf\#subsection.#1}{\S #1}}

\newcommand{\HTTthm}[2]{\href{https://www.math.ias.edu/~lurie/papers/HTT.pdf\#theorem.#2}{#1 #2}}
\newcommand{\HAthm}[2]{\href{https://www.math.ias.edu/~lurie/papers/HA.pdf\#theorem.#2}{#1 #2}}

\newcommand{\SAGthm}[2]{\href{https://www.math.ias.edu/~lurie/papers/SAG-rootfile.pdf\#theorem.#2}{#1 #2}}

\newcommand{\stacks}[1]{\cite[\stackstag{#1}]{stacksproject}}
\newcommand{\kerodon}[1]{\cite[\kerodontag{#1}]{kerodon}}
\newcommand{\HTT}[2]{\cite[\HTTthm{#1}{#2}]{HTT}}
\newcommand{\HA}[2]{\cite[\HAthm{#1}{#2}]{HA}}
\newcommand{\SAG}[2]{\cite[\SAGthm{#1}{#2}]{SAG}}

%% file: styles/macros.tex
\newcommand{\from}{\colon}

\newcommand{\ot}{\leftarrow}

\newcommand{\surj}{\twoheadrightarrow}
\DeclareRobustCommand\longtwoheadrightarrow {\relbar\joinrel\twoheadrightarrow}
\newcommand{\lsurj}{\longtwoheadrightarrow} 

\newcommand{\inj}{\hookrightarrow}
\newcommand{\hooklongrightarrow}{\lhook\joinrel\longrightarrow}
\newcommand{\linj}{\hooklongrightarrow}

\DeclarePairedDelimiter{\sqrbr}{[}{]}

\DeclarePairedDelimiter{\vertbr}{|}{|}

\DeclarePairedDelimiterX{\setcond}[2]{\{}{\}}{#1 \;\delimsize\vert\; #2}

\newcommand{\blank}{-}

\newcommand{\equivalent}{\simeq}

\newcommand{\andeq}{\text{\qquad and\qquad}}

\DeclareMathOperator{\pr}{pr}

\newcommand{\lang}{\longrightarrow}

\newcommand{\fin}{\operatorname{fin}}

\DeclareMathOperator{\ev}{\operatorname{ev}}

\renewcommand{\in}{\smallin}

\newcommand{\cartesian}{\arrow[dr, phantom, very near start, "{ \lrcorner }"]}
\newcommand{\diagr}[1]{\scriptstyle{(#1)}}

\newcommand{\cat}{\mathbf}
\newcommand{\catC}{\cC}
\newcommand{\catD}{\cD}
\newcommand{\catI}{\cI}
\newcommand{\catE}{\cE}

\newcommand{\op}{\operatorname{op}}
\newcommand{\lex}{\operatorname{lex}}

\newcommand{\ladj}{\dashv}
\newcommand{\vsim}{\wr}

\newcommand{\overcat}[2]{{#1}_{\!/\!#2}}

\newcommand{\undercat}[2]{{#1}_{\!#2/}}
\NewDocumentCommand{\Ani}{o}{\IfValueTF{#1}{\overcat{\cat{Ani}}{#1}}{\cat{A\hspace{-.2ex}ni}}}
\newcommand{\pifinAni}{\Ani_{\pi}}
\newcommand{\SigmafinAni}{\Ani_{\pi}}

\DeclareMathOperator{\Tw}{Tw}

\NewDocumentCommand{\SigmaAni}{O{\pi}}{\Ani_{#1}}
\NewDocumentCommand{\SigmapfAni}{O{\pi}}{\Pro(\Ani_{#1})}

\NewDocumentCommand{\pfAni}{o}{\Pro(\pifinAni)\IfValueT{#1}{_{/#1}}}
\NewDocumentCommand{\spfAni}{o}{\IfValueTF{#1}{\overcat{(\simpl{\pfAni})}{#1}}{\simpl{\pfAni}}}
\NewDocumentCommand{\truncpfAni}{O{1} o}{\SigmapfAni[\pi, \leq #1\IfValueT{#2}{/#2}]}

\newcommand{\ptpifinAni}{\Ani_{\pi, \terminal}}
\NewDocumentCommand{\ptpfAni}{o}{\SigmapfAni[\pi, \point \IfValueT{#1}{/#1}]}

\DeclareMathOperator{\sets}{\cat{Set}}
\newcommand{\finSets}{\sets_{\fin}}

\DeclareMathOperator{\Grp}{\cat{Grp}}
\newcommand{\finGrp}{\Grp_{\fin}}

\DeclareMathOperator{\Ab}{\cat{A\hspace{-.3ex}b}}
\newcommand{\finAb}{\Ab_{\fin}}

\newcommand{\ssets}{\cat{sSet}}

\DeclareMathOperator{\Hom}{Hom}

\DeclareMathOperator{\Aut}{Aut}

\DeclareMathOperator{\Pro}{Pro}

\DeclareMathOperator{\unit}{\eta}
\DeclareMathOperator{\counit}{\epsilon}

\DeclareMathOperator{\id}{id}

\DeclareMathOperator{\Fun}{Fun}

\newcommand{\isomto}{\xlongrightarrow{\,\smash{\raisebox{-1ex}{\ensuremath{\displaystyle\sim}}}\,}}
\newcommand{\isomfrom}{\xlongleftarrow{\,\smash{\raisebox{-1ex}{\ensuremath{\displaystyle\sim}}}\,}}

\DeclareSymbolFont{supplsymbols}{T1}{\familydefault}{m}{n}
\SetSymbolFont{supplsymbols}{bold}{T1}{\familydefault}{bx}{n}

\DeclareMathSymbol{\frenchlq}{\mathopen}{supplsymbols}{19}
\DeclareMathSymbol{\frenchrq}{\mathclose}{supplsymbols}{20}

\DeclareMathOperator*{\limit}{lim}
\DeclareMathOperator*{\colimit}{colim}

\makeatletter
  \def\varphlim@#1#2{%
    \vtop{\m@th\ialign{##\cr
      \hfil$#1\operator@font\phantom{lim}$\hfil\cr
      \noalign{\nointerlineskip\kern1.5\ex@}#2\cr
      \noalign{\nointerlineskip\kern-\ex@}\cr}}%
  }
  \def\varinjcolim{%
    \mathop{\ooalign{colim\cr\hfil$\mathpalette\varphlim@{\rightarrowfill@\textstyle}$\hfil}}\nmlimits@%
  }
\makeatother

\DeclareMathOperator{\cofilteredlim}{\limit}

\DeclareMathOperator{\prolimit}{\frenchlq \cofilteredlim \frenchrq}

\DeclareMathOperator{\rep}{\mathit{h}}

\newcommand{\pro}{\operatorname{pro}}
\newcommand{\cofilt}{\operatorname{cofilt}}
\DeclareMathOperator{\D}{\cD}

\DeclareMathOperator{\C}{\mathrm{C}}

\DeclareMathOperator{\RMod}{RMod}

\newcommand{\topos}{\mathscr}

\DeclareMathOperator{\Shape}{\Pi_{\infty}}
\NewDocumentCommand{\pfShape}{O{\blank}}{\pfcompl[\Shape(#1)]}

\DeclareMathOperator{\globsec}{\Gamma}

\DeclareMathOperator{\Ho}{ho}
\DeclareMathOperator{\h}{h}

\DeclareMathOperator{\nerve}{\operatorname{N}}
\DeclareMathOperator{\cechnerve}{\operatorname{\check{N}}}
\DeclareMathOperator{\trunc}{\tau}

\newcommand{\simpl}[1]{\operatorname{s}\! #1}
\NewDocumentCommand{\simplex}{o}{\mathbf{\Delta}\IfValueT{#1}{^{\! #1}}}
\NewDocumentCommand{\dsimplex}{O{n}}{\partial\!\simplex^{\! #1}}

\DeclareMathOperator{\map}{Map}
\newcommand{\Map}{\map}

\NewDocumentCommand{\htpycls}{O{\blank} O{\blank} o o}{\sqrbr{#1, #2}\IfValueT{#3}{_{#3}}\IfValueT{#4}{^{#4}}}

\DeclareMathOperator{\htpygrp}{\pi}
\newcommand{\conncomp}{\pi_0}
\newcommand{\connected}{>0}

\NewDocumentCommand{\pfcompl}{O{(\blank)} O{\pi}}{#1^{\wedge}_{#2}}

\newcommand{\terminal}{\ast}
\newcommand{\point}{\terminal}
\DeclareMathOperator{\fib}{fib}

\newcommand{\LoopAni}{\mathrm{\Omega}}
\DeclareMathOperator{\B}{\mathrm{B}}
\DeclareMathOperator{\K}{\mathrm{K}}

\NewDocumentCommand{\real}{O{\blank}}{\vertbr{#1}}

\NewDocumentCommand{\mat}{O{\blank}}{\vertbr{#1}}

\newcommand{\doubleslash}{{/\mkern-6mu/}}
\newcommand{\modmod}{\doubleslash}

\DeclareMathOperator{\sBar}{Bar}

\DeclareMathOperator{\Sup}{\mathrm{S}}

\newcommand{\et}{{\operatorname{\acute{e}t}}}

\DeclareMathOperator{\FEt}{F\acute{E}t}

\newcommand{\etsite}[1]{\operatorname{{{\acute{E}}t}_{#1}}}
\newcommand{\etsiteop}[1]{\operatorname{{{\acute{E}}t}_{#1}^\mathrm{op}}}
\NewDocumentCommand{\ettopos}{m o}{{#1}_{\et}\IfValueT{#2}{^{#2}}}
\newcommand{\ethyptopos}[1]{\ettopos{#1}[\hyp]}

\DeclareMathOperator{\etfdtlgrp}{\hat{\pi}_{1}}
\NewDocumentCommand{\ethtpygrp}{O{1} O{\blank}}{\hat{\pi}_{#1}(#2)}

\DeclareMathOperator{\ethtpytype}{\widehat{\Pi}}
\newcommand{\open}{\operatorname{op}}
\newcommand{\pioneop}{\pi_{\!1}\textnormal{-}\!\open}
\newcommand{\etapioneop}{\eta\textnormal{-}\pi_{\!1}\textnormal{-}\!\open}

\NewDocumentCommand{\ethtpycls}{O{\blank} O{\blank} o o}{\htpycls[\ethtpytype(#1)][\ethtpytype(#2)][\ethtpytype(#3)][#4]}
\NewDocumentCommand{\etopcls}{O{\blank} O{\blank} o}{\ethtpycls[#1][#2][#3][^{\pioneop}]}
\NewDocumentCommand{\etetaopcls}{O{\blank} O{\blank} o}{\ethtpycls[#1][#2][#3][^{\etapioneop}]}
\NewDocumentCommand{\etisocls}{O{\blank} O{\blank} o}{\ethtpycls[#1][#2][#3][^{\simeq}]}

\DeclareMathOperator{\HC}{\operatorname{HC}}
\DeclareMathOperator{\HCet}{\HC^{\et}}

\DeclareMathOperator{\opGal}{\operatorname{Gal}}

\newcommand{\absGal}[1]{\opGal_{#1}}
\DeclareMathOperator{\Galois}{Gal}

\DeclareMathOperator{\Stab}{Stab}

\newcommand{\out}{\operatorname{out}}

\DeclareMathOperator{\Sch}{\cat{Sch}}

\newcommand{\dom}{\operatorname{dom}}

\DeclareMathOperator{\Spec}{Spec}

\DeclareMathOperator{\HH}

\DeclareMathOperator{\Sh}{Shv}
\newcommand{\hyp}{\operatorname{hyp}}

\newcommand{\qcqs}{\operatorname{qcqs}}

\newcommand{\ph}{\varphi}

\newcommand{\rH}{H}

\newcommand{\bA}{{\mathbb A}}

\newcommand{\bN}{{\mathbb N}}

\newcommand{\bZ}{{\mathbb Z}}

\newcommand{\cC}{{\mathscr C}}
\newcommand{\cD}{{\mathscr D}}
\newcommand{\cE}{{\mathscr E}}

\newcommand{\cI}{{\mathscr I}}

\newcommand{\cX}{{\mathscr X}}

\newcommand{\dM}{{\mathcal M}}

\newcommand{\dO}{{\mathcal O}}

\newcommand{\fh}{{\mathfrak h}}

\newcommand{\defding}[1]{\emph{#1}}

%% file: content/introduction.tex
%----------------------------------------------------------------------------------------------------------------------------------
\subsection{Anabelian geometry and étale homotopy theory}

Traditional anabelian geometry compares morphisms between unpointed connected schemes with classes of homomorphisms between their étale fundamental groups. 
The latter is an algebraic invariant defined on connected schemes together with a choice of a base point.
Replacing the étale fundamental group $\etfdtlgrp$ by the étale fundamental groupoid  $\ethtpytype_{\le 1}$ solves this conceptual shortcoming: the comparison is by applying a functor defined on unpointed objects. 

Higher dimensional anabelian geometry naturally belongs to the realm of {étale homotopy theory}.
It studies the functor that assigns to a qcqs scheme $X$ its profinite étale homotopy type $\ethtpytype(X)$ of which the fundamental groupoid $\ethtpytype_{\le 1}(X)$ is the $1$-truncation. 
Beginning with Artin and Mazur \cite{AM}, several different incarnations of the étale homotopy type occur in the literature.

In this paper, we consider $\ethtpytype(X)$ as an object of the $\infty$-category $\pfAni$ of profinite anima.
The approach of working with profinite anima offers significant technical advantages.
Above all, there is better compatibility with base change and proétale descent, which we use in \Cref{sec:spreading-out-homotopies} to spread out homotopies.
To provide the required foundational background, we have compiled necessary definitions and facts on profinite homotopy theory  in an appendix.

In order to state our main result, we need to fix some notation.
For a profinite anima $B$, we denote the $\infty$-category of profinite anima over $B$ by $\pfAni[B]$.
We also write
\[
    \map_{B}(E', E) 
    = \map_{\pfAni[B]}(E', E)
    \andeq
    \htpycls[E'][E][B]    = \pi_{0}\! \map_{B}(E', E)
\]
for the anima of maps $E' \to E$ over $B$ and the corresponding set of homotopy classes.
Furthermore, we call a map $\varphi \colon E' \to E$  of profinite anima \defding{$\pi_1$-open} if for any point $e'$ of $E'$ with image $e = \varphi(e')$ in~$E$, the induced homomorphism of profinite groups $\pi_1(E', e') \to \pi_1(E, e)$ is open. The collection of all $\pi_1$-open maps $E' \to E$ forms an anima with a monomorphism, see \cref{subsec:monomorphisms}.
\[
 \map^{\pioneop}_{B}(E', E)  \linj \map_{B}(E', E),
\]
and its connected components form the subset of $\pi_1$-open homotopy classes
\[
\htpycls[E'][E][B]^{\pioneop} = \pi_0\! \map^{\pioneop}_{B}(E', E) = \{\varphi \in \htpycls[E'][E][B]  \ ; \  \varphi \text{ is $\pi_1$-open}\} \subseteq \htpycls[E'][E][B] .
\]

%----------------------------------------------------------------------------------------------------------------------------------
\subsection{The main result}
A field is \defding{sub-$p$-adic} if it has an embedding into a finitely generated field extension of~${\mathds Q}_p$ for some prime number $p$.
A \defding{hyperbolic curve} over a field $k$ is the complement $X$ of a divisor $D$ in a smooth projective geometrically connected curve $\smash{\overline{X}/k}$ of genus $g$ such that $D$ is étale over $k$ and $X$ has negative Euler characteristic $\chi(X) = 2 - 2g - \deg(D)<0$.
A \defding{relative curve over  $S$} is the complement $\smash{X = \overline{X} \smallsetminus D}$ of a relative étale divisor $D$ in a smooth proper relative curve $\smash{\overline{X} \to S}$ with geometrically connected fibres.
We call $X\to S$ \emph{hyperbolic} if the fibres are hyperbolic curves.

\begin{thmABC} \label{maintheorem}
Let $S$ be a connected, normal scheme of finite type over a sub-$p$-adic field  and let $X\to S$ be a relative hyperbolic curve.
Then, for every smooth connected  $S$-scheme $Y$ of finite type, surjective over $S$, and with geometrically connected generic fibre,  the natural map
\begin{equation} 
\label{eq:main map2}
	\ethtpytype \colon \Hom_{S}^{\dom}(Y, X) \longrightarrow \etopcls[Y][X][S]
\end{equation}
is  a bijection from the set of dominant $S$-morphisms to the set of homotopy classes of $\pi_1$-open maps over~$\ethtpytype(S)$.
\end{thmABC}

\begin{remarks} 
\begin{enumerate}
	\item \label{ex:counterexample-alex-jakob}
	The surjectivity assumption on $Y \to S$ in \Cref{maintheorem} is necessary.
	Let $k$ be a field of characteristic $0$ and let $S = \bA^1_k$. 
	Let $U \subset S$ be an open that is a hyperbolic curve over $k$. 
	We set $X = U \times_k S \to S$, the constant curve over $S$ with fibre $U$, 
	and $Y = S \times_k U \to U \inj S$, the constant curve over $U$ with fibre $S$ considered as an $S$-curve. 
	Then there is no dominant $S$-map $Y \to X$. However, the isomorphism $\tau \colon Y \to X$ 
	that interchanges the factors induces a $\ethtpytype(S)$-equivalence 
	$\ethtpytype(\tau) \colon \ethtpytype(Y)\to \ethtpytype(X)$ because $\tau$ is a $k$-morphism 
	and the projection $S \to \Spec(k)$ induces an equivalence 
	$\ethtpytype(S) \isomto \ethtpytype(\Spec(k))$ of profinite anima.
	\item
	We will prove a refined version of \cref{maintheorem}.
        In fact, the anima $\map^{\pioneop}_{\ethtpytype(S)}(\ethtpytype(Y), \ethtpytype(X))$ is discrete, 
        see \cref{maintheorem-sharp}.
	Hence, considering the set $\Hom_{S}^{\dom}(Y, X)$ as a discrete anima, 
	the map \eqref{eq:main map2} of \Cref{maintheorem} lifts to an equivalence 
	\[
		\ethtpytype \colon \Hom_{S}^{\dom}(Y, X) 
		\isomto 
		\map^{\pioneop}_{\ethtpytype(S)}\big(\ethtpytype(Y), \ethtpytype(X)\big) .
	\]
	This constitutes a strong rigidity principle for maps to the profinite étale homotopy type of a relative hyperbolic curve.
\end{enumerate}
\end{remarks}

\subsection{Outline of the proof}
The starting point of this paper is the observation to be explained in \cref{cor:Mochizuki-via-homotopy-types} that, in the special case $S=\Spec(k)$, \cref{maintheorem} is a reformulation of Mochizuki's theorem on the anabelian geometry of hyperbolic curves \cite{Mochizuki99}.
This reformulation is similar in spirit but not exactly the same as the reformulation given in \cite[Theorem 3.1]{Schmidt-Stix16}, where the étale homotopy type was considered as an object of the homotopy category of the pro-category of simplicial sets.
Then we proceed by reducing \cref{maintheorem} to this special case.

Let $\eta \in S$ be the generic point, and $X_\eta = X \times_S \eta$ the generic fibre of $X$.
Our strategy is to reduce  \Cref{maintheorem} to the  case $S=\eta$  via the commutative diagram
\begin{equation}
\label{eq:pre main square}
    \begin{tikzcd}
        \Hom_{S}(Y, X) \arrow[r] \arrow[d, hook, "\diagr{1}"']
        & \ethtpycls[Y][X][S] \arrow[d, hook, "\diagr{2}"] \\
        \Hom_{\eta}(Y_{\eta}, X_{\eta}) \arrow[r]
        & \ethtpycls[Y_{\eta}][X_\eta][\eta].
    \end{tikzcd}
\end{equation}
That the restriction to the generic fibre $\diagr{1}$ is injective, is an easy geometric fact. 
A major difficulty is the injectivity of $\diagr{2}$ that we show in \cref{sec:spreading-out-homotopies} using proétale descent, after establishing the mere existence of the arrow $\diagr{2}$ in \cref{sec:basechange} in the first place.
In \cref{sec:extending-curves}, we show that  \eqref{eq:pre main square} is cartesian: a morphism $g \colon Y_{\!\eta} \to X_\eta$ over $\eta$ extends to a morphism $f \colon Y \to X$ over $S$ if and only if the induced map on profinite homotopy types extends up to homotopy.
This extension follows from the main result of \cite{stix:monodromy-extension},  which generalises work by Moret-Bailly \cite{Moret-B-pur}.

Diagram  \eqref{eq:pre main square} needs a modification to be useful in the proof of \Cref{maintheorem}.
We say that  a map $\ethtpytype(Y) \to \ethtpytype(X)$ over $\ethtpytype(S)$ is \defding{$\eta$-$\pi_1$-open} if its restriction $\ethtpytype(Y_{\!\eta}) \to \ethtpytype(X_\eta)$ to $\ethtpytype(\eta)$ is $\pi_1$-open.
Now we restrict to dominant morphisms on the geometric side and to ($\eta$-)$\pi_1$-open homotopy classes of maps on the homotopy theoretic side:
\begin{equation}
\label{eq:main square}
    \begin{tikzcd}
        \Hom_{S}^{\dom}(Y, X) \arrow[r, "\diagr{4}"] \arrow[d, hook, "\diagr{1}"']
        & \ethtpycls[Y][X][S]^{\etapioneop} \arrow[d, hook, "\diagr{2}"] \\
        \Hom_{\eta}^{\dom}(Y_{\eta}, X_{\eta}) \arrow[r, "\diagr{3}"', "{\sim}"]
        & \ethtpycls[Y_\eta][X_\eta][\eta]^{\pi_1\text{-op}}.
    \end{tikzcd}
\end{equation}
The lower horizontal map $\diagr{3}$ is bijective by the reformulation of Mochizuki's theorem \cite[Theorem~A]{Mochizuki99} in terms of profinite anima, see \cref{cor:Mochizuki-via-homotopy-types}.
Since \eqref{eq:main square} is also cartesian,  this identifies the images of $\diagr{1}$ and $\diagr{2}$ via the isomorphism $\diagr{3}$. This shows the bijectivity of $\diagr{4}$,  which is the map of \cref{maintheorem} because the properties $\eta$-$\pi_1$-open and $\pi_1$-open are equivalent in this situation, as we show in \cref{sec:finalize-proof}. This finishes the proof of \cref{maintheorem}.

%----------------------------------------------------------------------------------------------------------------------------------
\subsection{Notation and conventions}
Throughout this paper we write \defding{qcqs} as an abbreviation for \lq\lq quasi-compact and quasi-separated\rq\rq.
We write $\Sch^{\qcqs}$ for the full subcategory of the category of schemes $\Sch$ spanned by qcqs schemes.
For qcqs schemes we frequently use noetherian approximation as in \cite[Appendix C]{TT90}, see also \stacks{01Z1}.

For an $S$-scheme $X$, the base change $X \times_S T$ is denoted by $X_T$. 
By the phrase \emph{\'{e}tale cover}  we mean finite \'{e}tale morphism, i.e., rev\^{e}tement \'{e}tale in the sense of \cite{SGA1}.
An \emph{\'{e}tale covering} is a surjective family of étale morphisms. 
Usually, we suppress \enquote{$\et$} from the notation.
For example, we write $\etfdtlgrp$ for the (profinite) étale fundamental group and similarly for higher étale homotopy groups, étale homotopy types, cohomology, and so on.

\smallskip
We freely use the language of $\infty$-categories as developed in \cite{HTT}, \cite{HA}, and \cite{SAG}.
In particular, given an ordinary $1$-category $\catC$, we consider it as an $\infty$-category via the fully faithful nerve construction $\nerve_{\bullet} \from \cat{Cat} \to \ssets$, see \cite[\kerodontag{002L}, \kerodontag{002Z}]{kerodon}.
Moreover, we usually suppress the application of $\nerve_{\bullet}$ from the notation and still write $\catC$ for the $\infty$-category $\nerve_{\bullet}\!\catC$.

We follow Clausen and Scholze's terminology and write $\Ani$ for the $\infty$-category of anima, given as the homotopy coherent nerve of the simplicial category of fibrant simplicial sets (i.e., what Lurie refers to as the $\infty$-category of \defding{spaces} or  of \defding{$\infty$-groupoids}).
The word \defding{anima} means an object (a vertex) of $\Ani$, i.e., a Kan complex in classical language.
The homotopy category $\h\!\Ani$ is precisely the usual homotopy category $\Ho(\ssets)$ of the ($1$-)category of simplicial sets with the Kan-Quillen model structure.

Let $\catC$ be an $\infty$-category and let $s \in \catC$ be an object (i.e., a vertex).
\begin{itemize}
        \item We write $\overcat{\catC}{s}$ (resp. $\undercat{\catC}{s}$) for the $\infty$-category of objects over $s$ (resp. under $s$).
        \item We usually write $\map_{s}(y, x)$ instead of $\map_{\overcat{\catC}{s}}(y, x)$ for the mapping anima in the overcategory.
        \item Given two maps (i.e., edges of the mapping anima) $f\!, g \from y \to x$ in $\overcat{\catC}{s}$, we sometimes write $f \simeq_{s} g$ if $f$ is homotopic to $g$ in $\map_{s}(y, x)$.

        \item A terminal object of $\catC$ will often be denoted by $\terminal$, and we refer to $\catC_{\terminal} = \undercat{\catC}{\terminal}$ as the \emph{$\infty$-category of pointed objects in $\catC$}.
        \item We write $\map_{\point}(y, x)$ instead of $\map_{\catC_{\terminal}}(y, x)$ for the anima of pointed maps in $\catC$.
\end{itemize}

%----------------------------------------------------------------------------------------------------------------------------------
\addtocontents{toc}{\SkipTocEntry}
\subsection*{Acknowledgments}
The authors acknowledge support by Deutsche Forschungsgemeinschaft  (DFG) through the Collaborative Research Centre TRR 326 ``Geometry and Arithmetic of Uniformized Structures'', project number 444845124.
TH was furthermore supported through the Walter Benjamin Programme (DFG, project number 565137331) and thanks IH\'ES for providing excellent working conditions.

%%% Local Variables:
%%% mode: LaTeX
%%% TeX-master: "../haupt"
%%% End:

%% file: content/profinite-anima.tex
In this section we recall the definition of the profinite étale homotopy type as an object of the procategory of $\pi$-finite anima and compare it with the classical construction of Artin and Mazur.
In particular, we compare the \'etale homotopy groups in \Cref{prop:profiniteness}.
Moreover, we survey  pro\'etale descent of the profinite étale homotopy type, which will be needed later in the proof in our \Cref{thm:spreading-out-homotopies} about spreading out homotopies.

%----------------------------------------------------------------------------------------------------------------------------------
\subsection{The étale homotopy type} 

We start by recalling the notion of \emph{shape}.

\begin{recollection} \label{rec:shape}
        \begin{thmlist}
            \item
            Following \cite[\HTTsubsec{7.1.6}]{HTT}, the \defding{shape} $\Shape\!\topos{X}$ of an $\infty$-topos $\topos{X}$ is the object of $\Pro(\Ani) \equivalent \Fun^{\lex}(\Ani,\Ani)^{\op}$ given by the  composite
            \begin{equation*}
                \globsec_{\! \topos{X}\!,\ast} \circ \globsec^{*}_{\! \topos{X}} \from \Ani \to \Ani,
            \end{equation*}
            where $\globsec_{\! \topos{X}\!,\ast} = \Gamma(\cX,-)$ is the global section functor on $\topos{X}$ and $\globsec^{*}_{\topos{X}}$ its left adjoint, the constant sheaf functor.
            That is, for each anima $B$, there is a natural equivalence
            \[
                \map(\Shape(\topos{X}), B) \simeq \globsec(\topos{X}, \globsec^{*}_{\topos{X}}B).
            \]

            \item
            Given a geometric morphism $f_{\!*} \from \topos{X} \to \topos{Y}$ with unit $ \unit \from \id_{\topos{Y}} \to f_{\!*} \circ f^{*}$, there is an induced map of proanima $\Shape(\topos{X}) \to \Shape(\topos{Y})$ given by the map
            \begin{equation*}
                \globsec_{\topos{Y}, *} \unit \globsec^{*}_{\topos{Y}} \from \globsec_{\topos{Y}, *} \globsec^{*}_{\topos{Y}} \longrightarrow \globsec_{\topos{Y}, *} f_{\!*} f^{*} \globsec^{*}_{\topos{Y}} \equivalent \globsec_{\topos{X}, *} \globsec^{*}_{\topos{X}}.
            \end{equation*}
        \end{thmlist}
       \end{recollection}

\begin{recollection}
    \label{rec:etale-sheaves}
    Let $X$ be a qcqs scheme. We denote by $\etsite{X}$ the (classical, $1$-categorical) small étale site of $X$. Let $\catC$ be any $\infty$-category.
    \begin{thmlist}
        \item A functor $F \from\! \etsiteop{X} \to \catC$ is an \defding{étale sheaf on $X$ with values in $\catC$} if the following two conditions are satisfied.\smallskip
            \begin{enumerate}[(i)] 
                \item
                The functor $F$ preserves finite products.
                \item \label{defitem:descent}
                For any $V \in \etsite{X}$ and étale surjection $U_{0} \surj V$ with \v{C}ech nerve $U_{\!\bullet}$ the diagram $F \circ U_{\!\bullet}$ is a limit diagram, i.e.,
                      \[
                            F(V) \simeq \lim (\begin{tikzcd}[cramped, sep=small]
                                F(U_{0})
                                    \arrow[r, shift left = 0.25em]
                                    \arrow[r, leftarrow]
                                    \arrow[r, shift right = 0.25em]
                                  &  F(U_{0} \times_{V} U_{0})
                                      \arrow[r, shift left = 0.5em]
                                      \arrow[r, leftarrow, shift left = 0.25em]
                                      \arrow[r]
                                      \arrow[r, leftarrow, shift right = 0.25em]
                                      \arrow[r, shift right = 0.5em]
                                    &  \ldots
                            \end{tikzcd}) .
                      \]
            \end{enumerate} 
        \item 
        A functor $F \from\! \etsiteop{X} \to \catC$ is an \defding{étale hypersheaf on $X$ with values in $\catC$} if it is an étale sheaf that satisfies condition \ref{defitem:descent} above for any étale hypercovering of $V \in \etsite{X}$, not just \v{C}ech nerves of étale coverings.
        \item 
        We write $\Sh_{\et}(X, \catC)$ (resp.\ $\Sh^{\hyp}_{\et}(X, \catC)$) for the full subcategories of $\Fun(\etsiteop{X}, \catC)$ spanned by the étale sheaves (resp.\ the étale hypersheaves).
        \item We call $\ettopos{X} = \Sh_{\et}(X, \Ani)$ (resp.\ $\ethyptopos{X} = \Sh_{\et}^{\hyp}(X, \Ani)$) the \defding{$\infty$-topos of étale sheaves on $X$} (resp.\ \defding{étale hypersheaves on $X$}).
        By \HTT{}{6.5.3.12}, $\ethyptopos{X}$ is the \defding{hypercompletion} of the $\infty$-topos~$\ettopos{X}$.  
            \end{thmlist}
\end{recollection}

We recall the definition of the category $\Pro(\Ani_\pi)$ of profinite anima and the profinite completion functor in \cref{app:profinite-anima}.

\begin{definition}
\label{def:etale-homotopy-type}
Let $X$ be a qcqs scheme. 
\begin{deflist}
	\item 
	We call $\Pi(X) =\Shape(X_{\et}^{\mathrm{hyp}}) \in \Pro(\Ani)$ the \defding{\'etale homotopy type} of $X$.
    	Any morphism $f \colon X \to S$ of qcqs schemes induces a map $\Pi(f) \from \Pi(X) \to \Pi(Y)$.  
    	\item 
	We refer to its profinite completion    
        \[
        		\ethtpytype(X) = \pfcompl[\Pi(X)] \in \Pro(\Ani_\pi)
        \]
        as the \defding{profinite \'etale homotopy type} of $X$.
        Any morphism $f \colon X \to S$ of qcqs schemes  induces a map
        \[
          	\ethtpytype(f) = \pfcompl[\Pi(f)] \from \ethtpytype(X) \to \ethtpytype(S).
        \]

    	\item 
    	Any geometric point $x \from \Spec(\Omega) \to X$ induces a point $\point_x$ of $\ethtpytype(X)$ by functoriality
    	\[
    		\ethtpytype(x) \colon \ethtpytype(\Spec(\Omega)) \to \ethtpytype(X).
    	\]
    	By means of $\point_x$ we define for every integer $n \geq 0$
         \[
           	\ethtpygrp[n][X,x] = \pi_n(\ethtpytype(X),\point_x), 
         \]
         the \defding{$n$-th profinite \'etale homotopy group} of $X$ with base point $x$. 
\end{deflist}
\end{definition}

\begin{remark}
The first profinite étale homotopy group of a pointed, connected scheme $\etfdtlgrp(X, x)$ is 
canonically isomorphic to the étale fundamental group of $(X, x)$ as defined in \cite{SGA1}.
\end{remark}

\begin{remark}[Sheaves versus hypersheaves] Let $B$ be a truncated anima. Then the constant étale sheaf associated with $B$ is truncated by \HTT{}{5.5.6.16} and hence already a hypersheaf by \HTT{}{6.5.2.9}.
Therefore the shapes $\Pi_\infty(X_{\et})^\wedge_\pi$ and $\pfShape[\ethyptopos{X}]$ corepresent the same functor $\Ani_\pi\to \Ani$, i.e.,
the natural map
$\Pi_\infty (\ethyptopos{X}) \to \Pi_\infty (\ettopos{X}) $ induces an equivalence on profinite completions
\[
\ethtpytype(X)=\Pi(X)^\wedge_\pi=\pfShape[\ethyptopos{X}] \isomto \pfShape[\ettopos{X}].
\]
\end{remark}
  
\begin{recollection}
  \label{rec:continuity-of-the-etale-homotopy-type} \label{rec:proetale-hyperdescent}
  The profinite étale homotopy type functor
  \[
    \ethtpytype \colon \Sch^{\qcqs} \longrightarrow \pfAni
  \]
  preserves cofiltered limits of schemes with affine transition maps, see e.g. \cite[Prop. 3.7]{HHW24}.
  In particular, the  profinite étale homotopy type is compatible with absolute noetherian approximation of qcqs schemes. 

Moreover, recall that the proétale topology of Bhatt and Scholze \cite{BS15} is a Grothendieck topology on $\Sch^{\qcqs}$. 
By \cite[Cor.\,2.41]{haine2025a}, the profinite completion of the shape of the hypercomplete proétale topos coincides with the profinite étale homotopy type.
Hence  $\ethtpytype$ is a hypercomplete proétale cosheaf, i.e., for any semi\-simplicial proétale hypercovering $U_{\!\bullet} \to X$ the induced diagram 
\[
	\ethtpytype(U_{\!\bullet}) \longrightarrow \ethtpytype(X)
\]
is a colimit diagram in $\Pro(\pifinAni)$. 
\end{recollection}

Finally, we note that the profinite étale homotopy type is a geometric invariant, at least in characteristic zero.

\begin{proposition}[{\cite[Example 4.18, Corollary 4.26]{HHW24}}]
\label{prop:geometric-invariant} 
Let $K/k$ be an extension of separably closed fields, $X$ a qcqs scheme over~$k$ and $X_K$ its base change to~$K$. Assume that $k$ is of characteristic zero or that $X$ is proper over~$k$.
Then the natural map
\[
	\ethtpytype(X_K) \longrightarrow \ethtpytype(X)
\]
is an equivalence of profinite anima.
\end{proposition}

%----------------------------------------------------------------------------------------------------------------------------------
\subsection{Comparison with Artin-Mazur's étale homotopy type} 
\label{rem:comparison-shape-artin-mazur}

A scheme $X$  is \emph{étale locally connected} if every scheme $U$ étale over $X$ is a coproduct of connected schemes.
A locally noetherian scheme is étale locally connected by \cite[\nopp 6.1.9]{EGAI}.
We let $\HC^{\et}_{}(X)$ be the category of étale hypercoverings of~$X$.
This category is simplicially enriched and cofiltered as a simplicial category \cite[\S5]{Hoyois}.

We consider the homotopy category $\Ho(\ssets)$ of simplicial sets with the Kan-Quillen model structure.
The \emph{étale homotopy type functor of Artin and Mazur} \cite{AM} assigns to any étale locally connected scheme $X$ an object $$\Pi^{\mathrm{AM}}(X) \in \Pro(\Ho(\ssets)).$$
The pro-object $\Pi^{\mathrm{AM}}(X)$ 
is indexed by the (cofiltered) homotopy category $\h\!\HC^{\et}_{}(X)$ and the functor sends an étale hypercovering $U_{\!\bullet} \to X$ to the homotopy type of the simplicial set ${\scriptstyle \Pi}(U_{\!\bullet})$, where ${\scriptstyle \Pi}$ is the connected component functor. 

If ${x}$ is a geometric point of $X$, the same construction yields the pointed version
\[
\Pi^{\mathrm{AM}}(X,{x}) \in \Pro(\Ho(\ssets_*))
\]
in the procategory of the homotopy category of pointed simplicial sets. Moreover, Friedlander \cite{Fr82} constructs a natural representative 
$
  \Pi^{\mathrm{Fr}}(X)\in \Pro(\ssets)
$
of $\Pi^{\mathrm{AM}}(X) \in \Pro(\Ho(\ssets))$. 
The \emph{Artin-Mazur étale homotopy groups} $\pi_n^{\mathrm{AM}}(X,{x})$ are the pro-groups $\pi_n\big(\Pi^{\mathrm{AM}}(X,{x}))$.

\begin{remark}
By \cite[Corollary 10.7]{AM}, the first Artin-Mazur étale homotopy (pro-)group $\pi_1^{\mathrm{AM}}(X,{x})$ is 
canonically isomorphic to 
the \emph{groupe fondamental élargi of $(X, x)$}, defined in \cite[X~\S6]{SGA3}.
\end{remark}

We have the following comparison results for homotopy types and homotopy groups.

\begin{proposition}[{\cite[Proposition~5.1]{Hoyois}}] 
\label{hoyois-comp}
Let $X$ be an étale locally connected qcqs scheme (e.g., $X$ noetherian).
Then the étale homotopy type $\Pi(X)\in \Pro(\Ani)$ is corepresented by the simplicially enriched cofiltered diagram
\[
      \HCet(X) \xlongrightarrow{{\scriptstyle \Pi}} \ssets,\quad  U_{\bullet} \mapsto {\scriptstyle \Pi} (U_{\bullet}).
\]
Consider the natural functor of (1-) categories:
\[
	\h(\Pro (\Ani)) \longrightarrow \Pro(\h\!\Ani)) = \Pro(\Ho(\ssets)).
\]
Then $\Pi(X)$ on the left maps to $\Pi^{\mathrm{AM}}(X)$ on the right. 
\end{proposition}

In particular, the homotopy (pro)\-groups of $\Pi(X)$ coincide with the Artin-Mazur homotopy groups of $X$. If $X$ is connected and geometrically unibranch (e.g., normal), the Artin-Mazur étale homotopy pro-groups are profinite by \cite[Theorem 11.1]{AM}. 

\begin{proposition} 
\label{prop:profiniteness} 
Let $(X,{x})$ be a geometrically pointed, connected, étale locally connected qcqs scheme which is geometrically unibranch. 
Then, for all $n$,  we have natural isomorphisms
\[
	\pi_n^{\mathrm{AM}}(X,{x})= \pi_n(\Pi(X),*_x)\isomto \hat\pi_n(X, x).
\]
\end{proposition}

\begin{proof} 
By \cite[Theorem 11.1]{AM} and \cref{hoyois-comp}, $\Pi(X,{x})\in \Pro(\Ani_*)$ is corepresented by a diagram $\prolimit_i T_{\!i}$ of the form  $\catI \to \Ani_*$, $i\mapsto T_{\!i}$, with connected pointed anima $T_{\!i}$ with finite homotopy groups.
Since the $T_{\!i}$'s are not necessarily truncated, $\Pi(X)$ is not necessarily an object of $\Pro(\Ani_{\pi,*})$. 
Nevertheless, the profinite completion $T^\wedge_\pi$ of $T$ is then prorepresented by the system 
\[
	T^\wedge_\pi = \prolimit_{i,j} \tau_{\le j} T_{\! i}
\]
of the various truncations (coskelata).
Since $\pi_n(T_{\!i}) = \pi_n(\tau_{\le j} T_{\! i})$ for $j\ge n$, we obtain
\[
	\pi_n^{\mathrm{AM}}(X,{x}) = \prolimit_i \pi_n(T_{\! i}) = \prolimit_{i,j} \pi_n(\tau_{\le j}T_{\! i})= \hat\pi_n(X, x). 
	\qedhere
\]
\end{proof}

%%% Local Variables:
%%% mode: LaTeX
%%% TeX-master: "../haupt"
%%% End:

%% file: content/K-pi-1-schemes.tex
The main protagonists of our result in anabelian geometry are hyperbolic curves.
We will explain in  \cref{sec:BG} that hyperbolic curves are $\K(\pi,1)$-spaces. 
This is well known at various technical and terminological levels, but we formulate and prove it here for the profinite \'etale homotopy type as a profinite anima and with respect to the natural definition in \'etale homotopy theory.
This plays a crucial role in \cref{sec:spreading-out-homotopies}, where we extend homotopies from the generic fibre to the entire scheme.

%----------------------------------------------------------------------------------------------------------------------------------
\subsection{Finite étale covers and classifying anima}

We start by relating maps to a classifying anima $\B\!G$ and finite \'etale covers. 
For the necessary results and notation on classifying anima we refer to \cref{subsec:BG}.  In particular, we denote by 
$E\modmod G$ the \emph{homotopy quotient} of $E$ by $G$.

\begin{proposition} 
\label{lem:finite-etale-galois-cartesian-square}
Let $(X, x)$ be a pointed connected qcqs scheme, and let $f \from Y \to X$ be a finite \'etale cover.
Let $\ph \colon \etfdtlgrp(X, x) \surj G$ be a finite quotient such that the action of $\etfdtlgrp(X, x)$ on the fibre $Y_{\!x}$ factors through $G$.
Then there is a natural cartesian square 
\[
	\begin{tikzcd}
      	\ethtpytype(Y) \arrow[d, "\ethtpytype(f)"'] \arrow[r, "\psi"] 
	&  \ethtpytype(Y_{\!x}) \modmod G  \arrow[d,"\iota"] 
	\\
      	\ethtpytype(X) \arrow[r, "\ph"]  
	& \B\!G ,
    	\end{tikzcd}
\]
where $\iota$ is induced by the $G$-equivariant constant map 
$Y_{\!x} = f^{-1}( x) \to  x$ in view of $\B\!G = \, \point\! \modmod G$.
\end{proposition}
\begin{proof} 
Since $X$ is connected and $Y \to X$ is finite \'etale, $Y$ is a finite disjoint union of finite connected \'etale covers of $X$. The assertion is compatible with finite disjoint unions in $Y$, and thus we may assume that $Y$ is also connected. 

We next fix a point $ y \in Y_{\! x}$.
Since $Y$ is connected, $G$ acts transitively on $Y_{\! x}$ with stabilizer $H \subseteq G$ of $y$ identified with the image of $\psi = \ph \circ \pi_1(f) \colon \etfdtlgrp(Y, y) \to \etfdtlgrp(X, x) \to G$.
Because $Y_{\!x}$ is a finite disjoint union of geometric points, we have an equivalence $\ethtpytype(Y_{\!x}) = G/H$, and therefore $\ethtpytype(Y_{\!x})  \modmod G = (G/H)\modmod G = \B\!H$.
Hence the square of the proposition can be written in the form 
\[
	\begin{tikzcd}
      	\ethtpytype(Y) \arrow[r,"\psi"] \arrow[d, "\ethtpytype(f)"'] 
	& \B\!H\arrow[d] 
	\\
     	\ethtpytype(X) \arrow[r, "\ph"] 
	& \B\!G.
    	\end{tikzcd}
\]
We start by reducing the assertion to the case $H = 1$.
Let $(Z, z)$ denote the pointed finite \'etale cover of $X$ associated with the action of $\etfdtlgrp(X, x)$ on $G$ via $\ph$.
By construction, we have maps $Z \to Y \to X$, and the image of $\etfdtlgrp(Z, z)$ in $H$, resp.\ in $G$, is trivial.
Assuming the trivial case (twice), in the diagram
\[
	\begin{tikzcd}
      	\ethtpytype(Z) \arrow[r] \arrow[d] 
	& \ethtpytype(Y) \arrow[r] \arrow[d] 
	& \ethtpytype(X) \arrow[d] 
	\\
      	\point \arrow[r] 
	& \B\!H \arrow[r] 
	& \B\!G
    	\end{tikzcd}
\]
the left and total rectangle are cartesian. Therefore, the right rectangle is also cartesian by pasting of pullbacks, see \Cref{cor:cancellation of pullbacks}~\ref{coritem:pullback-cancellation}.

We may now assume $H = 1$, i.e., the cover $Y \to X$ is Galois with Galois group $G$. 
 By \SAG{}{E.6.5.1}, the profinite variant of \cref{rec:pointed-vs-unpointed} \ref{homquot}, we may equivalently show that the action of $G = \opGal(Y/X)$ on $\ethtpytype(Y)$ induced by the evident action from geometry induces an equivalence
\[
	\ethtpytype(Y)\modmod G \isomto \ethtpytype(X).
\]
To this end, note that using the constant group scheme $G_X = \bigsqcup_G X$,  we can identify
\[
 	\cechnerve_X(Y)_k = \bigtimes_X^{k+1}Y  = Y \times_X \bigtimes_X^{k} G_{X}  = \bigsqcup_{G^k} Y .
\]
Since the \'etale homotopy type preserves coproducts, this shows that we can identify the homotopy type of the \v Cech-nerve
\[
	\ethtpytype(\cechnerve_X(Y)_\bullet) = \sBar_G(\ethtpytype(Y), \point)_{\bullet}
\]
with the two-sided Bar construction $\sBar_G$.
By \SAG{}{E.6.4.1} and \SAG{}{E.6.4.3}, the colimit of $\sBar_G(\ethtpytype(Y), \point)_{\bullet}$ computes   
$\ethtpytype(Y) \modmod G$.
Now observe that $f \from Y \to X$ is an étale cover of the connected scheme $X$, hence an étale covering. We conclude that the canonical map
\[
	\colimit_{\simplex^{\op}} \ethtpytype(\cechnerve_X(Y)) \isomto \ethtpytype(X)
\]
is an equivalence as desired.
\end{proof}

\begin{definition}
  \label{def:etale-n-truncated}
  Let $X$ be a qcqs scheme and $n \geq -1$.
  \begin{deflist}
    \item We write $\ethtpytype_{\leq n}(X)$ for the truncation $\trunc_{\leq n}\! \ethtpytype(X)$ and refer to $\ethtpytype_{\leq 1}(X)$ as the \emph{profinite étale fundamental group\-oid} of $X$.
    \item \label{defitem:etale-n-truncated} We say $X$ is \emph{profinite étale $n$-truncated} if $\ethtpytype(X)$ is an $n$-truncated object of $\Pro(\pifinAni)$.
  \end{deflist}
\end{definition}

\begin{proposition}
  \label{prop:finite-etale-cartesian-squares}
  Let $f \from Y \to X$ be a finite \'etale morphism of qcqs schemes.
  Then the square
  \[
    \begin{tikzcd}
      \ethtpytype(Y) \arrow[r, "{ \ethtpytype(f)}"] \arrow[d] & \ethtpytype(X) \arrow[d] \\
      \ethtpytype_{\leq 1}(Y) \arrow[r, "{\ethtpytype_{\leq 1}(f)}"'] & \ethtpytype_{\leq 1}(X)
    \end{tikzcd}
  \]
   is cartesian, where the vertical maps are the natural unit maps.
\end{proposition}

\begin{proof}
  By noetherian approximation, we write $f = \limit_i f_{\! i}$ as a cofiltered limit of finite \'etale morphisms $f_{\! i} \from Y_{\! i} \to X_i$ between schemes of finite type over $\bZ$.
  Then, by \Cref{rec:continuity-of-the-etale-homotopy-type}, we have that also $\ethtpytype(f) = \limit_i  \ethtpytype(f_{\!i})$.
  Since pullbacks as well as $\trunc_{\leq 1}$ commute with cofiltered limits, we are thus reduced to the case that $Y$ and $X$ are noetherian.
  In this case, both $Y$ and $X$ have finitely many connected components.

  Since $\ethtpytype(X) \to \ethtpytype_{\leq 1}(X)$ and $\ethtpytype(Y) \to \ethtpytype_{\leq 1}(Y)$ are bijective on the respective sets of connected components, we can treat each connected component separately and hence assume both $Y$ and~$X$ to be connected.

  Let ${y}$ be a geometric point of $Y$, write $x$ for its image under $f$ and endow $\ethtpytype(Y)$ and $\ethtpytype(X)$ with the induced points.
  We write $\etfdtlgrp(X, x) = \prolimit_\alpha G_\alpha$ as a surjective projective system of finite quotients $\ph_\alpha \colon  \etfdtlgrp(X, x)  \to G_\alpha$, which we may moreover assume to all be large enough such that the action of $\etfdtlgrp(X, x)$ on the fibre $f^{-1}( x)$ factors through $G_\alpha$. Let $i_\alpha: H_\alpha \to  G_\alpha$ denote the inclusion of the image of $\psi_\alpha = \ph_\alpha \circ \etfdtlgrp(f) \colon \etfdtlgrp(Y, y) \to \etfdtlgrp(X, x) \to G_\alpha$.
  Then $\etfdtlgrp(Y, y) = \prolimit_\alpha H_\alpha$ and for all $\alpha$ 
\[
    \begin{tikzcd}
      \ethtpytype(Y) \arrow[r, "\ethtpytype(f)"] \arrow[d, "\psi_\alpha", swap] & \ethtpytype(X) \arrow[d, "\ph_\alpha"] \\
      \B\!H_\alpha \arrow[r, "{\B\! i_\alpha}"'] & \B\!G_\alpha 
    \end{tikzcd}
\]
is cartesian by \cref{lem:finite-etale-galois-cartesian-square}. Since pullbacks commute with cofiltered limits, we obtain a cartesian diagram
\[
    \begin{tikzcd}
      \ethtpytype(Y) \arrow[r, "\ethtpytype(f)"] \arrow[d] & \ethtpytype(X) \arrow[d] \\
      \B\! \etfdtlgrp(Y, y)  \arrow[r, "{\B\!\etfdtlgrp(f)}"'] & \B\!\etfdtlgrp(X, x)  .
    \end{tikzcd}
\]
 Since $\B\!\etfdtlgrp(Y, {y})$ and $\B\!\etfdtlgrp(X, x)$ are $1$-truncated, these maps factor through $\ethtpytype_{\leq 1}(Y)$ and $\ethtpytype_{\leq 1}(X)$, respectively.
 As moreover both $Y$ and $X$ are connected, the canonical maps~$\ethtpytype_{\leq 1}(Y) \to \B\!\etfdtlgrp(Y, {y})$ and $\ethtpytype_{\leq 1}(X) \to \B\!\etfdtlgrp(X, x)$ are  equivalences.
 This finishes the proof.
\end{proof}

%----------------------------------------------------------------------------------------------------------------------------------
\subsection{\texorpdfstring{Profinite étale $\K(\pi,1)$ spaces}{Profinite étale K(π,1) spaces}}

\begin{definition}  
  \label{def:profinite-etale-K-pi-1}
  We say that a qcqs scheme $X$ is a \defding{profinite \'etale $\K(\pi, 1)$} if $X$ is connected and profinite \'etale $1$-truncated in the sense of \Cref{def:etale-n-truncated}~\labelcref{defitem:etale-n-truncated}.
\end{definition}

\begin{definition}
  \label{def:universal-cover}
  Let $X$ be a connected qcqs scheme and $x \to X$ a geometric point.
We write $\FEt_{(X, x)}$ for the category of \emph{pointed} connected finite étale covers of $X$.
The \emph{universal cover} of $(X, x)$ is the limit
      \[
        \widetilde{X}^{(x)} = \limit_{(Y, {y}) \in \FEt_{(X, x)}} Y.
      \]
     We often suppress the base point from the notation and simply write $\widetilde{X}$.
\end{definition}

The object $\widetilde{X}$ is endowed with a canonical geometric point $\tilde{x}$ of $\widetilde{X}$.
The pair (pro)represents the fibre functor $F_{\!x}$ in $ x$ on the Galois category of all finite étale covers of $X$.
Thus
\[
  \etfdtlgrp(X, x) = \Aut(F_{ x}) = \Aut^{\mathrm op}(\widetilde X/X),
\]
and so $\widetilde{X} \to X$ is naturally a $\etfdtlgrp(X, x)$-torsor.

\begin{lemma}
  \label{lem:compatibility-universal-cover}
  Let $X$ be a connected qcqs scheme and let $x \to X$ be a geometric point.
  Then the universal cover $\widetilde{X} \to X$ induces a fibre sequence of profinite anima
  \[
    \begin{tikzcd}
      \ethtpytype(\widetilde{X}) \arrow[r] \arrow[d] \cartesian & \ethtpytype(X) \arrow[d] \\
      \point \arrow[r, "{x}"'] & \ethtpytype_{\leq 1}(X)
    \end{tikzcd}
  \]
\end{lemma}

\begin{proof}
  Using that $\trunc_{\leq 1}$, the étale homotopy type, and pullbacks commute with cofiltered limits (see \Cref{cor:n-truncation-profinite-anima} and \Cref{rec:continuity-of-the-etale-homotopy-type}, respectively), we obtain a pullback square
  \[
    \begin{tikzcd}
      \ethtpytype(\widetilde{X}) \arrow[r] \arrow[d] & \ethtpytype(X) \arrow[d] \\
      \ethtpytype_{\leq 1}(\widetilde{X}) \arrow[r] & \ethtpytype_{\leq 1}(X)
    \end{tikzcd}
  \]
  by using \Cref{prop:finite-etale-cartesian-squares} and passing to the limit with respect to $(Y, {y}) \in \FEt_{(X, x)}$.
Therefore it suffices to show that $\ethtpytype_{\leq 1}(\widetilde{X})$ is contractible, i.e., by the profinite Whitehead theorem \ref{thm:profinite-whitehead-theorem} that $\widetilde{X}$ is simply connected.
This holds by the very construction of $\widetilde{X}$.
\end{proof}

The following proposition collects various useful alternative characterisations of the property \enquote{profinite \'etale $\K(\pi, 1)$}.

\begin{proposition}
\label{prop:criterionKpi1}
  Let $X$ be a connected qcqs scheme with a geometric point $ x$ and universal cover $\widetilde{X}$.
  Then the following are equivalent.
\begin{thmlist}
	\item
	\label{propitem:criterionKpi1classifyingmap}
	The canonical map $\ethtpytype(X) \to \B\!\etfdtlgrp(X, x)$ is an equivalence.
	\item
	\label{propitem:criterionKpi1Kpi1}
	$X$ is a profinite étale $\K(\pi,1)$.
	\item
	\label{propitem:criterionKpi1contractibleuniversalcover}
	$\widehat\Pi(\widetilde{X})$ is contractible.
	\item
	\label{propitem:criterionKpi1cohomology}
	For all finite abelian groups $A$ and all $i \geq 1$ the cohomology group $\rH^i(\widetilde{X},A)$ vanishes.
	\item
	\label{propitem:criterionKpi1colimitofcohomology}
	For all finite abelian groups $A$ and all $i \geq 2$ we have
	\[
	\colimit_{(Y, {y}) \in \FEt_{(X, x)}} \rH^i(Y,A) = 0.
	\]
\end{thmlist}
\end{proposition}

\begin{proof}
\lcref{propitem:criterionKpi1classifyingmap} $\Leftrightarrow$ \lcref{propitem:criterionKpi1Kpi1}: The canonical map factors as
\[
  \ethtpytype(X) \to \ethtpytype_{\leq 1}(X) \to \B\!\etfdtlgrp(X, x).
\]
The second map is a map between $1$-truncated objects that is an isomorphism on $\pi_0$ and $\pi_1$, hence is an equivalence by  the profinite Whitehead theorem, see \cref{thm:profinite-whitehead-theorem}. 

\lcref{propitem:criterionKpi1Kpi1} $\Leftrightarrow$ \lcref{propitem:criterionKpi1contractibleuniversalcover}:
By \cref{lem:compatibility-universal-cover}, the homotopy fibre of $\ethtpytype(X) \to \ethtpytype_{\leq 1}(X)$ in $ x$ is given by $\ethtpytype(\widetilde{X})$.
The claim thus follows from the fact that a map between connected profinite anima is an equivalence if and only if its homotopy fibre is contractible.

\lcref{propitem:criterionKpi1contractibleuniversalcover} $\Leftrightarrow$ \lcref{propitem:criterionKpi1cohomology}:
Note that $\ethtpytype(\widetilde{X})$ is simply connected.
Therefore, by \cref{lem:profinite-hurewicz-theorem}, we have canonical isomorphisms
\[
  \rH^{2}(\widetilde{X}, A) = \rH^2(\ethtpytype(\widetilde{X}), A) = \Hom(\ethtpygrp[2][X, x], A)
\]
for any choice of geometric point $x$ of $X$.
Thus \lcref{propitem:criterionKpi1cohomology} implies that $\ethtpygrp[2][X, x]$ vanishes for any choice of geometric point $x$ of $X$.
Repeating the same argument with $\rH^{i}$, by induction on $i \geq 2$, we see that $\ethtpygrp[i][\widetilde{X}, x] = 0$ for all $i \geq 1$.
Hence $\ethtpytype(\widetilde{X})$ is contractible by  \cref{thm:profinite-whitehead-theorem}.

\lcref{propitem:criterionKpi1cohomology} $\Leftrightarrow$ \lcref{propitem:criterionKpi1colimitofcohomology}: For qcqs schemes, étale cohomology is compatible with limits showing
\[
  \rH^i(\widetilde{X},A) = \colimit_{(Y, {y}) \in \FEt_{(X, x)}} \rH^i(Y,A).
\]
Note that the colimit for $i=1$ always vanishes, because an element $\alpha  \in \rH^1(Y,A)$ classifies an étale $A$-torsor $Z \to Y$, a connected component $Z^0$ of which kills the class: $\alpha|_{Z^0} = 0$.
\end{proof}

\begin{proposition}
  \label{prop:hyperbolic-curves-are-K-pi-1}
  Let $k$ be a field and let $X$ be a smooth connected curve over $k$ that is not proper of genus $0$.
  Then the following holds.
  \begin{thmlist}
    \item
    \label{propitem:fields-are-k-pi-1}
    $\Spec(k)$ is a profinite étale $\K(\pi, 1)$.

    \item
    \label{propitem:hyperbolic-curves-are-K-pi-1}
    $X$ is a profinite étale $\K(\pi, 1)$.

    \item
    \label{propitem:comparison-classifying-anima}
    The choices of a separable closure $k \subset \bar{k}$ and a geometric point $ x$ of $X$ determine equivalences
    \[
      \ethtpytype(k) \isomto \B\!\absGal{k} \hbox{ and }\ \ethtpytype(X) \isomto  \B\!\etfdtlgrp(X, x)
    \]
    of profinite anima. Here $\Galois_k = \Galois(\bar k/k)$ is the absolute Galois group of $k$ with base point $\bar k$.
  \end{thmlist}
\end{proposition}
\begin{proof}
In view of \cref{prop:criterionKpi1}, assertion \lcref{propitem:comparison-classifying-anima} follows from \lcref{propitem:fields-are-k-pi-1} and \lcref{propitem:hyperbolic-curves-are-K-pi-1}. It remains to verify condition \lcref{propitem:criterionKpi1colimitofcohomology} of \cref{prop:criterionKpi1} for $\Spec(k)$ and for $X$. For $\Spec(k)$ this follows since étale cohomology of a field is Galois cohomology of its absolute Galois group. For $X$ we refer to the proof of \cite[Proposition 15]{Schmidt-ext}. For affine curves, this follows already from \cite[Exp IX, Corollaire 5.7]{SGA4}, the result on étale cohomological dimension of affine curves.
\end{proof}

\begin{corollary}
  \label{cor:etale-htpy-classes-K-pi-1}
  Let $X$ and $Y$ be connected schemes over a field $k$.
  Assume that $X$ is a profinite étale $\K(\pi, 1)$.
  Choose a separable closure $\bar{k}$  of~$k$ and geometric points
  $x$ of $X$ and ${y}$ of $Y$ extending $\bar k$.
  \begin{thmlist}
    \item \label{coritem:equivalence}
    There is an equivalence
    \[
      \map_{\ethtpytype(k)}\big(\ethtpytype(Y), \ethtpytype(X)\big) 
      \simeq 
      \map_{\B\!\absGal{k}}\big(\B\!\etfdtlgrp(Y, {y}), \B\!\etfdtlgrp(X, x)\big).
    \]
    \item \label{coritem:bijection}
    There is a  bijection
    \[
        \ethtpycls[Y][X][k][] \isomto \htpycls[\B\!\etfdtlgrp(Y, {y})][\B\!\etfdtlgrp(X, x)][\B\!\absGal{k}][].
    \]
  \end{thmlist}
In both assertions the classifying anima on the right hand side are considered as unpointed objects.
\end{corollary}

\begin{proof}
Assertion \labelcref{coritem:equivalence} follows from the equivalence $\ethtpytype(X) \isomto  \B\!\etfdtlgrp(X, x)$ of \Cref{prop:hyperbolic-curves-are-K-pi-1}, the universal property of the $1$-truncation of $\ethtpytype(Y)$ because $\B\!\etfdtlgrp(X, x)$ is $1$-truncated, and the fact that $\ethtpytype_{\leq 1}(Y) \to \B\!\etfdtlgrp(Y, y)$ is an equivalence.

Assertion \labelcref{coritem:bijection} follows from \labelcref{coritem:equivalence} by passing to connected components of the mapping anima.
\end{proof}

%%%% Local Variables:
%%%% mode: LaTeX
%%%% TeX-master: "../haupt"
%%%% End:

%% file: content/mochizuki-anima-version.tex
In this section we define the map $\diagr{3}$ of diagram \eqref{eq:main square}, namely
\[
\Hom_{\eta}^{\dom}(Y_{\eta}, X_{\eta})  \stackrel{\diagr{3}}{\longrightarrow} \etopcls[Y_\eta][X_\eta][\eta],
\]
and show that  Theorem A of \cite{Mochizuki99} translates into $\diagr{3}$ being bijective.

Let $Y/k$ be connected and let $X/k$ be a geometrically connected profinite étale $\K(\pi, 1)$.
Then, for any choice of geometric points $y \to Y$ and $x \to X$, the following diagram commutes:
\begin{equation}
  \label{eq:comparison-mzki0}
  \begin{tikzcd}
           \Hom_{k}(Y, X)
           \arrow[rr, "{\ethtpytype}"]  \arrow[d, dashed, "\etfdtlgrp"']
           && \ethtpycls[Y][X][k] \arrow[d, "{\vsim}"'] \\
           \Hom^{\out}_{\absGal{k}}\big(\etfdtlgrp(Y, y), \etfdtlgrp(X, x)\big)
           && { \sqrbr{ \ethtpytype(Y), \B\!\etfdtlgrp(X, x) }_{\B\!\absGal{k}} }
           \arrow[ll, "{\sim}", "{\etfdtlgrp \modmod \etfdtlgrp(X_{\bar{k}}, x)}"',swap].
  \end{tikzcd}
\end{equation}
The right vertical arrow is a bijection since $X$ is a profinite étale $\K(\pi, 1)$, see \Cref{cor:etale-htpy-classes-K-pi-1}.
The bottom map is bijective by  \cref{cor:htpy-classes-of-maps-over-BG} since  $X$ is geometrically connected and hence the map $\etfdtlgrp(X, x) \to \absGal{k}$ is surjective.

If we further restrict to dominant morphisms $f \colon Y \to X$ of geometrically unibranch schemes, then the induced maps $\etfdtlgrp(f)$ are open.
It follows that $\ethtpytype(f)$ is a $\pi_1$-open map of homotopy types. We therefore obtain a commutative subdiagram
\begin{equation}
  \label{eq:comparison-mzki}
  \begin{tikzcd}
           \Hom_{k}^{\dom}(Y, X)
           \arrow[rr, "{\ethtpytype}"]  \arrow[d, dashed, "{\etfdtlgrp}"']
           && \etopcls[Y][X][k] \arrow[d, "\vsim"'] \\
           \Hom^{\open, \out}_{\absGal{k}}\big(\etfdtlgrp(Y, y), \etfdtlgrp(X, x)\big)
           && { \sqrbr{ \ethtpytype(Y), \B\!\etfdtlgrp(X, x)}_{\B\!\absGal{k}}^{\pioneop} }
           \arrow[ll,"{\pi_1 \modmod \etfdtlgrp(X_{\bar{k}}, x)}"', "\sim",swap].
  \end{tikzcd}  
\end{equation}
This enables us to reformulate \cite[Theorem A]{Mochizuki99} in terms of profinite étale homotopy types.
We also take this opportunity to weaken the assumption on $Y$ from smooth to normal of finite type.

\begin{theorem}[Homotopy-theoretic reformulation of Mochizuki's Theorem]
  \label{cor:Mochizuki-via-homotopy-types}
  Let $k$ be a sub-$p$-adic field, $X$ a hyperbolic curve over $k$ and\/ $Y$ a normal connected variety over $k$.
  Then the canonical map from dominant morphisms to homotopy classes of $\pi_1$-open maps of profinite anima
  \[
      \ethtpytype \colon \Hom_{k}^{\dom}(Y, X) \isomto \ethtpycls[Y][X][k][\pioneop]
  \]
  is a bijection.
\end{theorem}

\begin{proof}
  We first assume that $Y$ is smooth over $k$. Since, by \Cref{prop:hyperbolic-curves-are-K-pi-1} \labelcref{propitem:hyperbolic-curves-are-K-pi-1}, $X$ is a geometrically connected profinite étale $\K(\pi, 1)$, we have the commutative square \eqref{eq:comparison-mzki} above.
  The upper horizontal arrow is a bijection since the left vertical arrow is a bijection by \cite[Theorem~A]{Mochizuki99}.

If $Y$ is merely normal, then we may choose a resolution of singularities $\sigma \colon  {Y'} \to Y$ with $Y'$ smooth over $k$ and $\sigma$ proper. Precomposition with $\sigma$ induces the vertical maps in the commutative diagram
\[
  \begin{tikzcd}
           \Hom_{k}^{\dom}(Y', X)  \arrow[r, "{\ethtpytype}"', "\sim" , swap]
           & \etopcls[{Y'}][X][k] \arrow[r, "\sim"]
           & \Hom^{\open, \out}_{\absGal{k}}\big(\etfdtlgrp(\tilde{Y}, y), \etfdtlgrp(X, x)\big) \\
           \Hom_{k}^{\dom}(Y, X)    \arrow[r, "{\ethtpytype}"]   \arrow[u, hook, "-\circ \sigma"]
           & \etopcls[Y][X][k]  \arrow[r, "\sim"]   \arrow[u, "-\circ \ethtpytype(\sigma)"]
           & \Hom^{\open, \out}_{\absGal{k}}\big(\etfdtlgrp(Y, y), \etfdtlgrp(X, x)\big)   \arrow[u, hook, "-\circ \etfdtlgrp(\sigma)"] .
  \end{tikzcd}
\]
Note that $\etfdtlgrp(\sigma)$ is surjective, hence the right vertical map is injective.
The horizontal maps with the exception of the lower left one are bijections  because of the smooth case already dealt with and the discussion around diagram \eqref{eq:comparison-mzki}.
It remains to show that the outer rectangle is cartesian.
This means that we need to show that a dominant morphism ${f'} \colon {Y'} \to X$ 
factors through $\sigma$ to a morphism $f \colon Y \to X$ if the induced map $\etfdtlgrp(f') \colon \etfdtlgrp({Y'}) \to \etfdtlgrp(X)$ factors through $\etfdtlgrp(\sigma)$  and a map $\etfdtlgrp(Y) \to \etfdtlgrp(X)$.
Because $\sigma \colon  Y' \to Y$ is proper birational and $Y$ is normal, we have $\sigma_\ast \dO_{{Y'}} = \dO_Y$ and $f'$ factors if and only if $f'|_C$ is constant for all proper smooth $k$-curves $C$ mapping to $Y'$ with image entirely in a fibre of~$\sigma$.
The induced map $\etfdtlgrp(C_{\bar k}) \to \etfdtlgrp({Y'}_{\bar k}) \to \etfdtlgrp(X_{\bar k})$ factors through $\etfdtlgrp(\sigma)$ and thus through $\etfdtlgrp(\sigma(C_{\bar k})) = 1$.
But a morphism from a smooth curve to a hyperbolic curve which is trivial on geometric fundamental groups is constant.
This proves the contraction of $f'(C)$ to a point and hence the theorem.
\end{proof} 

%%% Local Variables:
%%% mode: LaTeX
%%% TeX-master: "../haupt"
%%% End:

%% file: content/basechange.tex
The notion of a \emph{quasifibration of topological spaces} was introduced by Dold and Thom in \cite{dold1958}. In this paper, we will develop the concept of \emph{quasifibrations of schemes}, but only as far as we need it for our application. We show that a relative hyperbolic curve over a normal scheme in characteristic zero is a quasifibration. This leads to the base change map, the arrow $\diagr{2}$ in the diagram \eqref{eq:main square} of the introduction. 
A comprehensive analysis of quasifibrations of schemes --- including positive and mixed characteristics, as well as non-normal base schemes --- will be the content of \cite{QFiber}.

%----------------------------------------------------------------------------------------------------------------------------------
\subsection{Quasifibrations}
\label{subsec:quasifibrations}

\begin{definition}
\label{def:quasifibration}
A  morphism $f \colon X \to S$ of qcqs schemes is called \defding{quasifibration}
if for all geometric points $s$ of $S$, the natural map of 
the fibre to the homotopy fibre induces an equivalence
 \[
 	\ethtpytype(X_{s}) \isomto 
	\fib_{\ethtpytype(s)}\big(\ethtpytype(f) \from \ethtpytype(X) \to \ethtpytype(S)\big).
\]
\end{definition}
 
For a more concise notation we will write $\fib_{s}$ for $\fib_{\ethtpytype(s)}$. 

\medskip
We  start by showing that quasifibrations are compatible with cofiltered limits. Using noetherian approximation, this reduces the general case to that of noetherian schemes in the proofs below.

\begin{lemma}
\label{lem:quasifibrations-compatibility-with-limits}
Let $(f_{\! n} \from X_n \to S_n)_n$ be a cofiltered system of morphisms of qcqs schemes with affine transition maps and limit $f = \limit_n f_{\! n} \from X \to S$.
If all $f_{\! n}$'s are quasifibrations, so is $f$.
\end{lemma}

\begin{proof}
  Let $s \to S$ be a geometric point and write $s_n$ for the image of $s \to S$ under $S \to S_n$.
  Since $f = \limit_n f_{\! n}$, we also have $\ethtpytype(f) \simeq \limit_n  \ethtpytype(f_{\! n})$ by the continuity of the étale homotopy type \Cref{rec:continuity-of-the-etale-homotopy-type}.
  Since limits commute with limits, we see that $\fib_{s}(\ethtpytype(f)) \simeq \limit_n \fib_{s_n}(\ethtpytype(f_{\! n}))$.
  Moreover, profinite completion preserves cofiltered limits, hence
\[
     \fib_{s}(\ethtpytype(f))
      \simeq \limit_n(\fib_{s_n}(\ethtpytype(f_{\! n})))
      \simeq \limit_n (\ethtpytype(X_{n, s_n}))
      \simeq  \ethtpytype(\limit_n X_{n, s_n})
      \simeq \ethtpytype(X_{s}). 
      \qedhere
\]
\end{proof}

\begin{proposition}
	\label{prop:finite-etale-is-quasifibration}
	Finite étale morphisms are  quasifibrations.
\end{proposition}

\begin{proof}
By \cref{lem:quasifibrations-compatibility-with-limits} and noetherian approximation, we may assume that  $f \colon X \to S$ is a finite étale morphism between noetherian schemes. 
We may further assume that $S$ is connected. 
Let $s \to S$ be a geometric point and $\ph \colon \etfdtlgrp(S, s) \surj G$ a finite quotient such that the action of $\etfdtlgrp(S, s)$ on the fibre $X_{s}$ factors through $G$.
Then, by \Cref{lem:finite-etale-galois-cartesian-square}, the right and the outer square in  the diagram 
\[
	\begin{tikzcd}
	\ethtpytype(X_{s})\rar\dar
	& \ethtpytype(X)\rar\dar  
	& \ethtpytype(X_{s}) \modmod G\dar
	\\
	\ethtpytype(s) \rar
	& \ethtpytype(S) \rar 
	& \B\!G
	\end{tikzcd}
\]
are cartesian.
By cancellation of pullbacks, \Cref{cor:cancellation of pullbacks}~\ref{coritem:pullback-pasting}, also the left square is cartesian.
\end{proof}

\begin{proposition}
\label{prop:quasifibrations-compatibility-et-htpytype-with-pullbacks}
We assume that in the cartesian square of qcqs schemes  
\[
	\begin{tikzcd}
	Y \rar{q} \arrow[d,"{g}"'] \cartesian & X \dar{f}
	\\
	T \rar{p} & S
	\end{tikzcd}
\]
both $g$ and $f$ are quasifibrations. 
Then also the following square of profinite anima is cartesian:
      \[
          \begin{tikzcd}
              \ethtpytype(Y) \arrow[r, "{\ethtpytype(q)}"] \arrow[d, "{\ethtpytype(g)}"'] \cartesian
                & \ethtpytype(X) \arrow[d, "{\ethtpytype(f)}"] \\
              \ethtpytype(T) \arrow[r, "{\ethtpytype(p)}"]
                & \ethtpytype(S) .
          \end{tikzcd}
      \]
\end{proposition}

\begin{proof}
By \Cref{lem:criterion-homotopy-pullback-fibres} it suffices to compare the homotopy fibres of $\ethtpytype(g)$  and $\ethtpytype(f)$.
Because of $\htpygrp_0(\ethtpytype(T)) = \ethtpygrp[0][T]$, it suffices to check this on points of $\ethtpytype(T)$ arising from geometric points $t$ of  $T$ and their images $s = p(t)$ in $S$.
The map comparing the homotopy fibres sits in a commutative diagram
\[
  \begin{tikzcd}
    \ethtpytype(Y_{\! t}) \arrow[d ]\arrow[r] & \ethtpytype(X_{s}) \arrow[d] \\
    \fib_t(\ethtpytype(g)) \arrow[r] &  \fib_{s}(\ethtpytype(f)). 
  \end{tikzcd}
\]
The vertical maps are equivalences since $g$ and $f$ are quasifibrations, respectively.
The upper horizontal map is an equivalence since $Y_{\!t} \to X_s$ is an isomorphism of schemes.
Consequently, also the lower horizontal map is an equivalence.
\end{proof}

%----------------------------------------------------------------------------------------------------------------------------------
\subsection{Geometric fibrations in characteristic zero are quasifibrations}
\label{subsec:Friedlander}

The goal of this section is to prove that geometric fibrations over a normal scheme in characteristic zero are quasifibrations.
\begin{definition}[\cite{Friedl-elfib}]
  \label{def:geometric-fibration}
    Let $f \from X \to S$ be a morphism of schemes.
    \begin{deflist}
        \item
        The map $f \from X \to S$ is said to be a \emph{special geometric fibration} if it admits a factorisation of the form
            \[
              \begin{tikzcd}
                X \arrow[rd, "f"'] \arrow[r, open, "j"] 
                & \bar{X} \arrow[d, "{\bar{f}}"] 
                & T \arrow[l, closed', "i"'] \arrow[ld, "{f|_{T}}"] 
                \\
                & S &
              \end{tikzcd}
            \]
            such that:
            \begin{deflist}
                \item $\bar{f} \from \bar{X} \to S$ is smooth and proper,
                \item 
                $\begin{tikzcd}[cramped, sep=small] i \from T \arrow[r, closed] & \bar{X} \end{tikzcd}$ 
                is a closed embedding with complement $\begin{tikzcd}[cramped, sep=small] j \from X \arrow[r, open] & \bar{X} \end{tikzcd}$, and 
               
                \item $T$ is the union of closed subschemes $T_{k}$ of pure codimension $c_{k}$ in $\bar{X}$ over $S$ such that each non-empty intersection $T_{k_{1}} \cap \ldots \cap T_{k_{s}}$ is smooth over $S$ of pure codimension $c_{k_{1}} + \ldots + c_{k_{s}}$.
            \end{deflist}
            \item The map $f \from X \to S$ is said to be a \emph{geometric fibration} if Zariski-locally on $S$ it is a special geometric fibration.
          \end{deflist}
\end{definition}

A relative curve is a special case of a geometric fibration.  

\begin{theorem}
\label{thm:generalisation-of-Friedlander}
Let $S$ be a normal qcqs scheme of characteristic zero.
Then every geometric fibration $f \from X \to S$ with geometrically connected fibres is a quasifibration.
\end{theorem}

\begin{remark} 
A generalisation of \Cref{thm:generalisation-of-Friedlander} with a notion of geometric fibration that includes \'etale locally special geometric fibrations to the case of non-normal $S$ and  in arbitrary characteristic is the subject of \cite{QFiber}.  
\end{remark}

\begin{proof}[Proof of \Cref{thm:generalisation-of-Friedlander}]
The result essentially follows from the fibre comparison for geometric fibrations proven by Friedlander \cite[Theorem 3.7]{Friedl-elfib}. 

By \cref{lem:quasifibrations-compatibility-with-limits} and noetherian approximation, we may assume that  $S$ is noetherian and connected. 
We consider the  cofiltered  simplicially enriched category $\HC^{\et}(S)$ of étale hypercoverings of $S$.
Then, by \cite[Theorem 5.1]{Hoyois}, the homotopy type $\Pi(S)$  is corepresented by the simplicially enriched cofiltered diagram
\[
	{\scriptstyle \Pi} \from \HCet(S) \longrightarrow \ssets,\quad U_{\!\bullet} \mapsto {\scriptstyle \Pi} U_{\!\bullet},
\]
and similarly for $X$.
For a morphism $f \colon X \to S$ there is a relative version $\HCet(f)$ (denoted by $J_{\!f}$ in \cite[Definition (1.1)]{Friedlander-fibel}) defined as follows.
\begin{itemize}
	\item 
	An object of $\HCet(f)$ is a triple 
	$(U_{\!\bullet} \to X, V_{\!\bullet} \to S, U_{\!\bullet} \to X \times_S V_{\!\bullet})$ 
	consisting of \'etale hypercoverings $U_{\!\bullet} \in \HC^{\et}(X)$, $V_{\!\bullet} \in \HC^{\et}(S)$ 
	and a morphism of \'etale hypercoverings $U_{\!\bullet} \to X \times_S V_{\!\bullet}$ in $\HC^{\et}(X)$.
    	We usually simply write $U_{\!\bullet} \to X \times_S V_{\!\bullet}$ for an object of $\HC^{\et}(f)$.
	\item
	A morphism $(U'_{\!\bullet} \to X \times_S V'_{\!\bullet}) \to (U_{\!\bullet} \to X \times_S V_{\!\bullet})$ is a pair $(U'_{\!\bullet} \to U_{\!\bullet}, V'_{\!\bullet} \to V_{\!\bullet})$  such that the obvious induced diagram commutes.
\end{itemize}
The category $\HC^{\et}(f)$ is a cofiltered simplicially enriched category in the sense of \cite[\S 5]{Hoyois}.
The functors 
\[
	\begin{tikzcd}[column sep=small]\HC^{\et}_{\infty}(X) 
	& \HC^{\et}_{\infty}(f) \arrow[l, "s"'] \arrow[r, "t"] 
	& \HC^{\et}_{\infty}(S)
	\end{tikzcd}
\]
induced on underlying $\infty$-categories by the natural source and target functors are cofinal, as can be seen by combining \cite[Lemma 5.4]{Hoyois} with \cite[Prop. (1.4)]{Friedlander-fibel}.
Therefore, the diagram
\[
   \begin{tikzcd}
     \HC^{\et}_{\infty}(f) \arrow[d, "s"'] \arrow[r, "t"] & \HC^{\et}_{\infty}(S) \arrow[d, "{\scriptstyle \Pi}"] \\
     \HC^{\et}_{\infty}(X) \arrow[r, "{\scriptstyle \Pi}"] & \Ani
   \end{tikzcd}
 \]
 induces a map $\Pi(X) \to \Pi(S)$ that, by the same argument as in the proof of \cite[Thm. 5.1]{Hoyois}, coincides with $\Pi(f)$.
 Let $s$ be a geometric point of $S$ and note that the image of $\fib_{s}(\Pi(f)) \in \Pro(\Ani)$ under the canonical functor $\Pro(\Ani) \to \Pro(\h\!\Ani)$ coincides with Friedlander's homotopy fibre $\fh(f_\et^r)$ of \cite[\S 2]{Friedlander-fibel}.
 It follows that 
 \[
    \trunc_{<\infty}\!\Pi(X_{s}) \to \trunc_{<\infty}\!\fib_{s}(\Pi(f)) = \fib_{s}(\trunc_{<\infty}\!\Pi(f))
 \]
 is an equivalence, as this can be checked on homotopy groups, which can be read off from the image under $\Pro(\Ani) \to \Pro(\h\!\Ani)$, where it is an immediate consequence of \cite[Thm. 3.7]{Friedl-elfib}.
 Consider the canonical commutative diagram
 \[
    \begin{tikzcd}
      \trunc_{<\infty}\!\Pi(X_{s}) \arrow[r] \arrow[d] & \trunc_{<\infty}\!\Pi(X) \arrow[rr, "{\trunc_{<\infty}\!\Pi(f)}"] \arrow[d] && \trunc_{<\infty}\!\Pi(S) \arrow[d] \\
      \ethtpytype(X_{s}) \arrow[r] & \ethtpytype(X) \arrow[rr, "{\ethtpytype(f)}"'] && \ethtpytype(S).
    \end{tikzcd}
 \]
Since $X_{s}$, $X$, and $S$ are noetherian and normal, their étale homotopy groups are already profinite by \Cref{prop:profiniteness}. 
Hence all vertical maps appearing in the above diagram are equivalences (see also \cite[Thm. 3.6.5]{DAGXIII} for a modern proof using shapes).
Summarizing, we see that
 \[
    \ethtpytype(X_{s}) = \trunc_{<\infty}\!\Pi(X_{s}) \isomto \fib_{s}(\tau_{<\infty}\Pi(f)) = \fib_{s}(\ethtpytype(f))
 \]
 is an equivalence as claimed.
\end{proof}

%----------------------------------------------------------------------------------------------------------------------------------
\subsection{The base change map}
\label{subsec:the-base-change-map}

As an application, we are able to define the arrow $\diagr{2}$ in the diagram \eqref{eq:pre main square} of the introduction.

\begin{construction}
\label{constr:basechange-for-maps-of-etale-homotopy-types}
Let $Y$ be a qcqs scheme over $S$ and let $S' \to S$ be a morphism of
qcqs schemes.
The projection maps of the fibre product $Y' = Y \times_S S'$  induce a map (unique up to homotopy)
\[
    c_Y \colon \ethtpytype(Y') \longrightarrow \ethtpytype(Y) \times_{\ethtpytype(S)} \ethtpytype(S').
\]
If $X \to S$ is a quasifibration such that also the base change $X' = X \times_S S'$ is a quasifibration, then by \Cref{prop:quasifibrations-compatibility-et-htpytype-with-pullbacks} the canonical map $c_X \colon \ethtpytype(X') \isomto \ethtpytype(X) \times_{\ethtpytype(S)} \ethtpytype(S')$ is an equivalence. This leads to a diagram of maps between mapping anima
\[
  \begin{tikzcd}
            \map_{\ethtpytype(S)} \big(\ethtpytype(Y), \ethtpytype(X) \big) \arrow[rr, "{\blank \times_{\ethtpytype(S)} \ethtpytype(S')}"]
            \arrow[d, dashed] && \map_{\ethtpytype(S')} \big(\ethtpytype(Y) \times_{\ethtpytype(S)} \ethtpytype(S'),\ethtpytype(X) \times_{\ethtpytype(S)} \ethtpytype(S') \big) \arrow[d, "{\blank \circ c_Y}"] \\
            \map_{\ethtpytype(S')} \big(\ethtpytype(Y'), \ethtpytype(X') \big)
             \arrow[rr, "{c_X \circ \blank}", "\sim"']
            && \map_{\ethtpytype(S')} \big(\ethtpytype(Y'), \ethtpytype(X) \times_{\ethtpytype(S)} \ethtpytype(S') \big)
  \end{tikzcd}
\]
\noindent
We call the (unique up to homotopy) dashed map the \emph{base change map} and denote it by
\[
  (\blank)_{S'} \colon  \map_{\ethtpytype(S)}\big(\ethtpytype(Y), \ethtpytype(X) \big) \longrightarrow  \map_{\ethtpytype(S')} \big(\ethtpytype(Y'), \ethtpytype(X') \big), \quad \ph \mapsto \ph_{S'} .
\]
We refer to $\ph_{S'}$ as the \emph{base change along $S' \to S$}.
Note that the base change of maps between étale homotopy types is only defined if the homotopy type of the target is compatible with base change.
\end{construction}

%%% Local Variables:
%%% mode: LaTeX
%%% TeX-master: "../haupt"
%%% End:

%% file: content/spreading-out-homotopies.tex
We start with an elementary geometric observation that yields the injectivity of the arrow $\diagr{1}$ in diagram \eqref{eq:pre main square} of the introduction.

\begin{lemma}
    \label{prop:injectivity-via-generic-fibre}
    Let $S$ be an irreducible scheme with generic point $\eta \in S$ and let $Y \to S$ be a scheme morphism such that  $Y$ is reduced and the generic fibre $Y_{\!\eta}$ is dense in $Y$.
    Let moreover $X\to S$ be a separated scheme morphism.
    Then the base change map
    \[
        \blank \times_{S} \eta \colon \Hom_{S}(Y, X) \longrightarrow{} \Hom_{\eta}(Y_{\!\eta}, X_{\eta})
    \]
    is injective.
\end{lemma}

\begin{proof}
  Since $X/S$ is separated, the equaliser $Z$ of  two $S$-maps $f,g \colon Y \to X$ is closed in $Y$.
  If $Z$ contains $Y_{\!\eta}$ then $f$ and $g$ agree because $Y_{\!\eta}$ is dense in $Y$ by assumption.
\end{proof}

The goal of this section is to prove the following \cref{thm:spreading-out-homotopies}, which is an analogue of \Cref{prop:injectivity-via-generic-fibre} for profinite étale homotopy types, and implies the injectivity of arrow $\diagr{2}$ in the diagram \eqref{eq:pre main square} of the introduction.
Recall \cref{subsec:monomorphisms} for the notion of monomorphisms of anima. 

\begin{theorem}[spreading out étale homotopies]
\label{thm:spreading-out-homotopies} 
Let $S'\to S$ be a birational morphism of normal noetherian schemes, $X\to S$ a morphism and
$Y \to S$ a flat and dominant morphism. Denote the induced base changes by $X' = X \times_S S'$ and $Y' = Y \times_S S'$ and assume that $Y$ and $Y'$ are normal and noetherian.

Moreover, assume that  $X \to S$ has the  property, that for every morphism $T\to S$ with $T$ normal and noetherian, the base change $X_T\to T$ is a quasifibration with profinite étale $1$-truncated geometric fibres. 
  
Then the base change map
  \[
     	(\blank)_{S'} \from \map_{\ethtpytype(S)}\big(\ethtpytype(Y), \ethtpytype(X)\big) 
     	\longrightarrow 
	\map_{\ethtpytype(S')}\big(\ethtpytype(Y'), \ethtpytype(X')\big)
  \]
  is a monomorphism of anima.
\end{theorem}

\begin{remark}
  \label{rem:why-spreading-out}
  Consider the special case where $S$ is normal and connected, $Y\to S$ is smooth, and $S' = \eta \to S$ is the inclusion of the generic point. Let $\ph, \psi \in \map_{\ethtpytype(S)}(\ethtpytype(Y), \ethtpytype(X))$.   Then \Cref{thm:spreading-out-homotopies} asserts that any homotopy $H_{\eta} \from \ph_{\eta} \to \psi_{\eta}$ over $\ethtpytype(\eta)$ between the induced maps on the profinite homotopy types of the generic fibres uniquely \emph{spreads out} to a homotopy $H \from \ph \to \psi$ over $\ethtpytype(S)$.
  This is why we refer to the above theorem as ``spreading out étale homotopies''.
\end{remark}

The proof of \Cref{thm:spreading-out-homotopies} will occupy the rest of this section, so let us explain the proof strategy first.  
Recall that a ring $A$ is \emph{$w$-contractible} if every faithfully flat ind-étale map $A \to B$ has a section, see  \cite[Definition 2.41]{BS15}. 
The key idea of the proof is to simplicially resolve $S$ by affine w-contractible schemes. 
This allows us to replace $S$ by an affine $w$-contractible $S$-scheme $\Spec(A)$ and $X$ by the respective base change, although the new $X \to S$ are typically not noetherian anymore. 
Thus $X$ becomes profinite \'etale $1$-truncated because, by assumption, $X \to S$ is relative profinite étale $1$-truncated. 
Eventually, the proof boils down to the fact that the inclusion of the generic point 
$\eta$ of a normal, connected scheme $S$ induces a surjection $\etfdtlgrp(\eta) \lsurj \etfdtlgrp(S)$ on étale fundamental groups.

\begin{lemma}\label{lem:w-contractibly-is-0-connective}
  If $S= \Spec(A)$ is a $w$-contractible affine scheme, then $\ethtpytype(S)$ is $0$-truncated, i.e., for every geometric point $s$ of $S$ the homotopy groups $\ethtpygrp[k][S, s]$ vanish for $k \geq 1$.
\end{lemma}

\begin{proof}
Let $Z=\Spec(B)$ be the connected component of $s$ in $S$.
Then $\ethtpygrp[k][Z,s] \to \ethtpygrp[k][S,s]$ is an isomorphism for $k\ge 1$. 

By \cite[Lemma 2.4.2]{BS15}, $A$ is w-strictly local, in particular, all local rings of $A$ at closed points are  strictly henselian by \cite[Lemma 2.2.9]{BS15}.
By \cite[Lemma 2.1.4]{BS15}, $Z$ contains a unique closed point $x$ of~$S$, hence $B$ is local.
Furthermore, $B = B_x$ is strictly henselian being a quotient of the strictly henselian ring $A_x$.
We conclude that $\widehat{\Pi}(Z)$ is contractible and thus its homotopy groups vanish.
\end{proof}

\begin{lemma}\label{eta-one-truncated}
Let $S=\Spec(A)$ be the spectrum of a normal domain $A$ with quotient field $Q(A)$ and let $\eta=\Spec(Q(A))\hookrightarrow S$ be its generic point. 
Furthermore, let $A\to B$ be an ind-étale ring homomorphism and $T=\Spec(B)$. 
Then $\ethtpytype(T_{\!\eta}) \in \Pro(\Ani_{\pi,\le 1})$, i.e., for every geometric point $t$ of $T_{\!\eta}$, we have
\[
  \ethtpygrp[n][T_{\!\eta},t] = 0 \quad \text{for } n\ge 2.
\]
Moreover, the natural homomorphism
\[
  \etfdtlgrp(T_{\!\eta},t) \longrightarrow \etfdtlgrp(T,t)
\]
is surjective and the natural map $\ethtpygrp[0][T_{\!\eta}] \rightarrow \ethtpygrp[0][T]$ is a bijection.
\end{lemma}

\begin{proof}
  By \Cref{rec:continuity-of-the-etale-homotopy-type}, all assertions are stable under limits.
  Hence we may assume that $B$ is an étale $A$-algebra. Then $T$ is the disjoint union of finitely many connected components which are spectra of normal domains, and $T_{\!\eta}$ is the disjoint union of their generic points.
  This shows the statement on $\hat\pi_0$, and the surjectivity of the map on $\etfdtlgrp$ is a well known fact, see \stacks{0BQI}.

Moreover, since spectra of fields are étale $K(\pi,1)$-spaces by \Cref{prop:hyperbolic-curves-are-K-pi-1}~\ref{propitem:fields-are-k-pi-1}, we conclude the vanishing of the higher homotopy groups of $T_{\!\eta}$. 
\end{proof}

\begin{lemma}\label{lem:epicrit}
Let $S=\Spec(A)$ be the spectrum of a normal domain $A$ with quotient field $Q(A)$ and let $\eta=\Spec(Q(A))\hookrightarrow S$ be its generic point. 
Furthermore, let $A\to B$ be an ind-étale ring homomorphism and $T=\Spec(B)$. 
Then the natural map
\[
    \ethtpytype(T_{\!\eta}) = \ethtpytype_{\le 1}(T_{\!\eta}) \longrightarrow\ethtpytype_{\le 1}(T)
\]
is an epimorphism in $\truncpfAni[1]$.
\end{lemma}

\begin{proof}
  We write $B=\colimit_\lambda B^\lambda$ as a filtered colimit of étale $A$-algebras $B^\lambda$ and set $T^\lambda=\Spec(B^\lambda)$.
  Using \Cref{rec:proetale-hyperdescent}, we obtain a limit decomposition
\[
  \ethtpytype_{\le 1}(T_{\!\eta}) = \limit_{\lambda \in \Lambda} \ethtpytype_{\le 1}(T_{\!\eta}^\lambda) \longrightarrow \limit_{\lambda \in \Lambda} \ethtpytype_{\le 1}(T^\lambda) = \ethtpytype_{\le 1}(T).
\]
Monomorphisms are stable under limits by \Cref{rem:monomorphisms}~\ref{remitem:monos-are-stable-under-limits}.
Thus, in order to verify that $\ethtpytype_{\le 1}(T_{\!\eta}) \to \ethtpytype_{\le 1}(T)$ is an epimorphism in $\Pro(\Ani_{\pi, \le 1})$, it suffices to test this against $\Gamma \in \Ani_{\pi, \leq 1}$.
By cocompactness of $\Gamma \in \Ani_{\pi, \leq 1}$, precomposition with $\ethtpytype_{\leq 1}(T_{\!\eta}) \to \ethtpytype_{\leq 1}(T)$ can be identified with
\[
	\colimit_{\lambda \in \Lambda^{\op}} \map\big(\ethtpytype_{\leq 1}(T^\lambda), \Gamma\big)
	\longrightarrow 
	\colimit_{\lambda \in \Lambda^{\op}} \map\big(\ethtpytype_{\leq 1}(T_{\!\eta}^\lambda), \Gamma\big).
\]
Since monomorphisms of anima are stable under filtered colimits by \Cref{rem:monomorphisms}~\ref{remitem:filtered-colimits-in-anima}, it thus suffices to show that for each $\lambda \in \Lambda$ the map $\ethtpytype_{\leq 1}(T_{\!\eta}^\lambda) \to \ethtpytype_{\leq 1}(T^\lambda)$ is an epimorphism in $\Pro(\Ani_{\pi, \leq 1})$.

At this point we have reduced the assertion to the special case where $A \to B$ is \'etale.
We may assume this now and drop the index $\lambda$. Note that now $T$ is the disjoint union of its connected components and $T_{\!\eta} \to T$ is bijective on connected components.
Using that monomorphisms are stable under products, again by \Cref{rem:monomorphisms}~\ref{remitem:monos-are-stable-under-limits}, we may work component by component.
Hence we furthermore assume that $T$ is connected.
We choose a geometric point $t \in T_{\!\eta}$ and denote its image in $T$ again by $t$.
Then, by the normality of $T$, the induced homomorphism on fundamental groups
\[
  \etfdtlgrp(T_{\!\eta}, t) \lsurj \etfdtlgrp(T, t)
\]
is surjective.
Since $\ethtpytype_{\leq 1}(T_{\!\eta}) \to \ethtpytype_{\leq 1}(T)$ can be identified with
$\B\!\etfdtlgrp(T_{\!\eta}, t) \to \B\!\etfdtlgrp(T, t)$, 
we  conclude by \Cref{cor:surjection-of-groups-induces-epimorphism-of-groupoids}.
\end{proof}

\begin{proof}[Proof of \Cref{thm:spreading-out-homotopies}] 
The proof begins with a few initial reduction steps. 
\begin{itemize}
	\item
	Step 1:  we may assume that $Y = S$.
\end{itemize}
Indeed, by assumption on  $X \to S$, the base change along $Y \to S$ and precomposition with the diagonal map $\ethtpytype(\Delta_{Y/S}) \colon \ethtpytype(Y) \to \ethtpytype(Y\times_S Y)$ induce an equivalence
\[
	\map_{\ethtpytype(S)}\big(\ethtpytype(Y), \ethtpytype(X)\big) 
	\longrightarrow 
	\map_{\ethtpytype(Y)}\big(\ethtpytype(Y), \ethtpytype(X \times_S Y)\big), 
	\quad \ph \mapsto \ph_Y \circ \ethtpytype(\Delta_{Y/S}).
\]
This yields a commutative diagram     
\[
        \begin{tikzcd}
          \map_{\ethtpytype(S)}\big(\ethtpytype(Y), \ethtpytype(X)\big) \arrow[r, "{(\blank)_{S'}}"] \arrow[d, "{\simeq}"'] 
          & \map_{\ethtpytype(S')}\big(\ethtpytype(Y'), \ethtpytype(X')\big) \arrow[d, "{\simeq}"] 
          \\
          \map_{\ethtpytype(Y)}\big(\ethtpytype(Y), \ethtpytype(X \times_S Y)\big) \arrow[r, "{(\blank)_{Y'}}"'] 
          & \map_{\ethtpytype(Y')}\big(\ethtpytype(Y'), \ethtpytype(X' \times_{S'} Y')\big),
        \end{tikzcd}
\]
where the vertical maps are the above equivalences.
With $X \to S$ also $X \times_S Y \to Y$ has the property that the base change to any normal scheme is a quasifibration. 
Moreover, since $Y \to S$ is flat, also $Y' \to Y$ is birational. We may therefore from now on assume that $Y$ equals $S$. 

\begin{itemize}
	\item
	Step 2:  we may further assume that $S' \to S$ is the inclusion $\eta \inj S$ of the disjoint union $\eta$ of generic points of $S$. 
\end{itemize}
Indeed, since $S' \to S$ is birational the morphism $\eta \inj S$ factors uniquely as $\eta \inj S'$ also describing the disjoint union of the generic points of $S'$. 
Assuming that the theorem holds for the inclusion of generic points, then in the commutative diagram
\[
	\begin{tikzcd}
         \map_{\ethtpytype(S)}\big(\ethtpytype(S), \ethtpytype(X)\big)
         \arrow[rr, "{(\blank)_{S'}}"]
         \arrow[dr,"{(\blank)_{\eta}}",swap]
         &&
         \map_{\ethtpytype(S')}\big(\ethtpytype(S'), \ethtpytype(X')\big)
         \arrow[dl,"{(\blank)_{\eta}}"]
	\\
        &
        \map_{\ethtpytype(\eta)}\big(\ethtpytype(\eta), \ethtpytype(X_{\eta})\big) .
        &
        \end{tikzcd}
\]
both maps $(\blank)_\eta$ are monomorphisms. Hence also $(\blank)_{S'}$ is a monomorphism by \cref{lem:cancellation-of-monos}. 

\begin{itemize}
	\item
	Step 3:  we may further assume that $S$ is affine.
\end{itemize}
Indeed, we can choose a Zariski-hypercovering $S_{\bullet} \to S$ such that each $S_k$ is the union of finitely many connected open affine subschemes of $S$, in particular $S_k$ is affine itself.  
We write $\eta_k = S_k \times_S \eta$ for the generic fibre of $S_k \to S$, i.e., the scheme of generic points of the normal affine $S_k$. Since $\ethtpytype(S) = \colimit_{k \in \Delta^{\op}} \ethtpytype(S_k)$, and similarly for $\ethtpytype(\eta)$, the vertical maps in the commutative diagram
\[
	\begin{tikzcd}
    	\map_{\ethtpytype(S)}\big(\ethtpytype(S), \ethtpytype(X)\big) \arrow[d,"\vsim"] \arrow[rrr] 
	&&& \map_{\ethtpytype(\eta)}\big(\ethtpytype(\eta), \ethtpytype(X_{\eta})\big) \arrow[d,"\vsim"] 
	\\
    	{\displaystyle \limit_{k \in \Delta}} \map_{\ethtpytype(S)}\big(\ethtpytype(S_k), \ethtpytype(X)\big) 
	\arrow[rrr, "{\lim_k(\blank \times_{\ethtpytype(S)} \ethtpytype(\eta))}"] 
	&&& {\displaystyle \lim_{k \in \Delta}} \map_{\ethtpytype(\eta)}\big(\ethtpytype(\eta_k), \ethtpytype(X_{\eta})\big)
  	\end{tikzcd}
\]
are equivalences. Therefore, since monomorphisms are stable under limits, it suffices to verify that 
\[
  \begin{tikzcd}
     \map_{\ethtpytype(S)}\big(\ethtpytype(S_k), \ethtpytype(X)\big) \arrow[rr, "{\blank \times_{\ethtpytype(S)} \ethtpytype(\eta)}"] 
     && \map_{\ethtpytype(\eta)}\big(\ethtpytype(\eta_k), \ethtpytype(X_{\eta})\big)
  \end{tikzcd}
\]
is a monomorphism for each $k \in \Delta$.
Repeating step 1, this reduces us to $S = S_k$ and $S' = \eta_k$, which is an instance of the assumptions that we reduce to in step 2.
This completes this reduction step.

\begin{itemize}
	\item
	Step 4:  we may further assume that $S$ is connected.
\end{itemize}
Indeed, since $S$ is noetherian and normal, it decomposes into the finite disjoint union of its connected components and the assertion of \Cref{thm:spreading-out-homotopies} can be shown componentwise.

\begin{itemize}
	\item
	Step 5:  finally we simplicially resolve $S$.
\end{itemize}
Using \cite[Lemma 2.4.9]{BS15}, we find a hypercovering $S_{\bullet} \to S$ such that each $S_k$ is affine, w-contractible and pro-étale over $S$.  
We consider the pullbacks $X_{\bullet} = X \times_{S} S_{\bullet} \to X$ and $S_{\bullet, \eta} = \eta \times_S S_{\bullet} \to \eta$.
Using that any base change of $f \from X \to S$ to a normal base scheme $T \to S$ is a quasifibration, we obtain the commutative diagram
\[
\begin{tikzcd}
\map_{\ethtpytype(S)}\big(\ethtpytype(S), \ethtpytype(X)\big)
\arrow[rrr, "{\blank \times_{\ethtpytype(S)} \ethtpytype(\eta)}"]
\arrow[dd, "{\blank \times_{\ethtpytype(S)} \ethtpytype(S_{\bullet})}"']
&&& \map_{\ethtpytype(\eta)}\big(\ethtpytype(\eta), \ethtpytype(X_{\eta})\big)
    \arrow[dd, "{\blank \times_{\ethtpytype(\eta)} \ethtpytype(S_{\bullet, \eta})}"] \\
&&& \\
\map_{\ethtpytype(S_{\bullet})}\big(\ethtpytype(S_{\bullet}), \ethtpytype(X_{\bullet})\big) \arrow[rrr, "{\blank \times_{\ethtpytype(S_{\bullet})} \ethtpytype(S_{\bullet, \eta})}"']
&&& \map_{\ethtpytype(S_{\bullet, \eta})}\big(\ethtpytype(S_{\bullet, \eta}), \ethtpytype(X_{\bullet} \times_{S_{\bullet}} S_{\bullet, \eta})\big).
\end{tikzcd}
\]
Note that $S_{\bullet} \to S$, $X_{\bullet} \to X$, and $S_{\bullet, \eta} \to \eta$ induce simplicial resolutions on étale homotopy types by \Cref{rec:proetale-hyperdescent}.
Therefore, both vertical maps are equivalences by \Cref{prop:geometric-realisation-of-pullbacks} \labelcref{propitem:fully-faithful}. 
Hence, if $\blank \times_{\ethtpytype(S_{\bullet})} \ethtpytype(S_{\bullet, \eta})$ is a monomorphism, then so is $\blank \times_{\ethtpytype(S)} \ethtpytype(\eta)$. 

It thus suffices to show that the horizontal map in
\[
\begin{tikzcd}
\map_{\ethtpytype(S_{\bullet})}\big(\ethtpytype(S_{\bullet}), \ethtpytype(X_{\bullet})\big)
\arrow[rrr, "{\blank \times_{\ethtpytype(S_{\bullet})} \ethtpytype(S_{\bullet, \eta})}"]
\arrow[drrr]
&&& \map_{\ethtpytype(S_{\bullet, \eta})}\big(\ethtpytype(S_{\bullet, \eta}), \ethtpytype(X_{\bullet} \times_{S_{\bullet}} S_{\bullet, \eta})\big) \arrow[d,"{\wr}"] \\
&&&
\map_{\ethtpytype(S_{\bullet})}\big(\ethtpytype(S_{\bullet, \eta}), \ethtpytype(X_{\bullet})\big)
\end{tikzcd}
\]
is a monomorphism.
Since the vertical map is an equivalence by adjunction, it suffices to consider the diagonal map, which is given by precomposition with $\ethtpytype(S_{\bullet, \eta}) \to \ethtpytype(S_{\bullet})$.
         
By \Cref{lem:w-contractibly-is-0-connective},  and since by assumption, any base change of $f \from X \to S$ to a normal base scheme $T\to S$ is a quasifibration with profinite étale 1-truncated geometric fibres, $\ethtpytype(X_{\bullet})$ lies in the full subcategory $\overcat{\simpl{\truncpfAni[1]}}{\ethtpytype(S_{\bullet})}$ of $\overcat{\simpl{\pfAni}}{\ethtpytype(S_{\bullet})}$.
Here $\simpl{\pfAni}$ denotes the $\infty$-category of simplicial objects in $\pfAni$, see \cref{subsec:simplicial-objects-and-resolutions}.

By \Cref{eta-one-truncated}, $\ethtpytype(S_{\bullet, \eta})$ is $1$-truncated and $\ethtpytype(S_{\bullet})$ is even $0$-truncated by \Cref{lem:w-contractibly-is-0-connective}. Hence, by \Cref{lem:right-adjoint-faithful},  it suffices to show that $\ethtpytype(S_{\bullet, \eta}) \to \ethtpytype(S_{\bullet})$ is an epimorphism in $\overcat{\simpl{\truncpfAni[1]}}{\ethtpytype(S_{\bullet})}$.

To this end, by \Cref{lem:epimorphisms-in-over-category}, it suffices to show that $\ethtpytype(S_{\bullet, \eta}) \to \ethtpytype(S_{\bullet})$ is an epimorphism in $\simpl{\truncpfAni[1]}$, which, by the dual of \Cref{rem:monomorphisms}, can be checked levelwise.
We have thus reduced the claim to showing that
\[
    \ethtpytype(S_{k, \eta}) \longrightarrow \ethtpytype(S_{k})
\]
is an epimorphism in $\truncpfAni[1]$ for every $k \in \simplex$. 
This follows from \Cref{lem:epicrit}.
\end{proof}

%%% Local Variables:
%%% mode: LaTeX
%%% TeX-master: "../haupt"
%%% End:

%% file: content/anabelian-criterion-for-extending-curves.tex
In this section we show that diagram \eqref{eq:pre main square} in \cref{sec:introduction} is cartesian.
Using the notation of \eqref{eq:pre main square}, we show more precisely that a $\kappa(\eta)$-map
$\smash{Y_{\!\eta} \to X_\eta}$ in the generic fibre extends to an $S$-map $Y \to X$ if and only if the induced map on étale homotopy types extends from the generic fibre to a  map over~$\ethtpytype(S)$.
In fact, we discuss a slightly more general set-up with the inclusion $\eta \to S$ of the generic point replaced by a dominant and birational morphism $S' \to S$.
But for the application to the proof of \cref{maintheorem} all that matters is the case $S' = \eta$.

\begin{proposition} [Spreading out morphisms]
\label{prop:extension of curves controlled by homotopy}
Let $h\colon  S' \to S$ be a birational morphism of irreducible, normal, excellent schemes, and let $p\colon  X \to S$ be a smooth hyperbolic curve. 

Let $q\colon  Y \to S$ be a dominant morphism of finite type with $Y$ irreducible and normal.
Then, for any $S'$-morphism $f' \from Y' = Y_{S'} \to X_{S'} = X'$, the following are equivalent.
\begin{enumerate}[label=(\alph*), ref=(\alph*)]
  	\item
	\label{propitem:spaces}
  	There is an $S$-morphism $f \colon Y \to X$ whose base change to $S'$ is $f'$
  	 \[
           \begin{tikzcd}
              Y' \cartesian \arrow[r,"h_Y"] \arrow[d, "f' "']
                & Y \arrow[d, "f",dotted] \\
              X' \arrow[r,"h_X"']
                & X.
            \end{tikzcd}
            \]

  	\item
	\label{propitem:homotopy types}
  	There is a $\ethtpytype(S)$-map $\psi \colon \ethtpytype(Y) \to \ethtpytype(X)$ which fits into a commutative diagram
  	 \[
           \begin{tikzcd}
              \ethtpytype(Y') \arrow[r, "{\ethtpytype(h_Y)}"] \arrow[d, "\ethtpytype(f')"']
                & \ethtpytype(Y) \arrow[d, "\psi",dotted] \\
              \ethtpytype(X') \arrow[r, "{\ethtpytype(h_X)}"']
                & \ethtpytype(X) .
            \end{tikzcd}
            \]
            \item
	\label{propitem:1-truncated homotopy types}
  	There is a $\ethtpytype_{\leq 1}(S)$-map $\ph \colon \ethtpytype_{\leq 1}(Y) \to \ethtpytype_{\leq 1}(X)$ which fits into a commutative diagram
  	 \[
           \begin{tikzcd}
              \ethtpytype_{\leq 1}(Y') \arrow[r, "{\ethtpytype_{\leq 1}(h_Y)}"] \arrow[d, "{\ethtpytype_{\leq 1}(f')}"']
                & \ethtpytype_{\leq 1}(Y) \arrow[d, "\ph",dotted] \\
              \ethtpytype_{\leq 1}(X') \arrow[r, "{\ethtpytype_{\leq 1}(h_X)}"']
                & \ethtpytype_{\leq 1}(X) .
            \end{tikzcd}
            \]
\end{enumerate}
\end{proposition}

\begin{proof}
Because $\ethtpytype$ is a functor $\Sch^{\qcqs} \to \pfAni$, the implication \labelcref{propitem:spaces} $\Rightarrow$ \labelcref{propitem:homotopy types} is immediate, and \labelcref{propitem:homotopy types} $\Rightarrow$ \labelcref{propitem:1-truncated homotopy types} is trivial.
So it remains to show \labelcref{propitem:1-truncated homotopy types} $\Rightarrow$ \labelcref{propitem:spaces}.
Because $Y' \to Y$ is also dominant, if $f$ as in \labelcref{propitem:spaces} exists, it is unique and automatically an $S$-morphism.

We claim that the statement is local with respect to finite étale covers $T \to S$ of the base.
Indeed, by (Galois) descent  and the stated uniqueness, it suffices to construct a map $f_{\!T} \colon  Y \times_S T \to X \times_S T$ which is compatible with the base change $f'_{\!T} \colon  Y' \times_S T \to X' \times_S T$.
But since assumption \labelcref{propitem:1-truncated homotopy types} together with \Cref{lem:finite-etale-galois-cartesian-square} also induces  a diagram
\[
     \begin{tikzcd}
              \ethtpytype_{\le 1}(Y' \times_S T) \arrow[r] \arrow[d, "{\ethtpytype_{\leq 1}(f'_{\!T})}"']
                & \ethtpytype_{\le 1}(Y \times_S T) \arrow[d, "{\ph_{T}}"] \\
              \ethtpytype_{\le 1}(X' \times_S T) \arrow[r]
                & \ethtpytype_{\le 1}(X \times_S T) ,
     \end{tikzcd}
\]
the existence of the claimed morphism $f_{\!T}$ follows, once we decompose $Y \times_S T$ and $X \times_S T$ into connected components, so that the assertion can be applied on the $T$-level.

We exploit the freedom to pass to a finite étale cover of $S$ to split the boundary divisor of $X/S$. 
The relative hyperbolic curve $X \to S$ admits by definition a
compactification $X \subseteq \overline{X}$ relative $S$ with a finite étale
boundary divisor $D = \overline{X} \setminus X$ of relative degree $n$.
By choosing a suitable connected finite étale cover $T \to S$ that
splits all connected components of  $D \to S$, we may assume, after base
change by $T \to S$ and resetting notation,  that $X$ is the complement
of the marked points of a relative hyperbolic curve of genus $g$ with
$n$ marked points over $S$.
Thus, $X/S$ is represented by a map $\xi_X \colon  S \to \dM_{g,n}$ to the moduli stack of smooth hyperbolic curves of type $(g,n)$.

Let $\eta$ be the generic point of $Y$. By assumption, the image $q(\eta)$ is the generic point of $S$ so that $\eta \to Y \to S$ lifts uniquely to a map $\eta \to S'$. The resulting $\eta \to Y'$ describes a generic point of an irreducible component of $Y'$ that birationally dominates $Y$. 
Since $X/S$ is of finite presentation, there is a dense open $U \subseteq Y$ such that
the restriction $h_{X} \circ f' \colon  Y' \to X' \to X$ to $\eta$ extends to a map $g\colon 
U \to X$ over $S$.
We obtain a commutative diagram
\[
     \begin{tikzcd}
          \eta \arrow[d, hook]  \arrow[r, hook]
          & Y' \arrow[dd] \arrow[r, "f'"]
          & X'  \arrow[d, "{h_{X}}"]
          &
          \\
          U \arrow[dr, open]  \arrow[rr, "g", near end, crossing over]
          && X \arrow[r]  \arrow[d]
          & \dM_{g,n+1} \arrow[d]
          \\
          & Y \arrow[r]
          & S \arrow[r, "{\xi_{X}}"]
          & \dM_{g,n}
    \end{tikzcd}	
\]
The bottom right square is cartesian. Therefore, in order to construct the map $f\colon  Y  \to X$, it suffices to show that the curve over $U$ corresponding to the composite map  $U \to X \to \dM_{g,n+1}$ extends to a curve of the same type on $Y$.
By \cite[Theorem~1.1]{stix:monodromy-extension},  this follows if we can show that we have an extension as maps of étale fundamental groups.
Applying the functor $\etfdtlgrp$ with appropriate base points, yields the commutative diagram
\[
     \begin{tikzcd}
          \etfdtlgrp(\eta) \arrow[d, twoheadrightarrow] \arrow[r]
          & \etfdtlgrp(Y') \arrow[dd] \arrow[r, "f'_{\!\ast}"]
          & \etfdtlgrp(X')  \arrow[d]
          &
          \\
          \etfdtlgrp(U) \arrow[dr,  twoheadrightarrow]  \arrow[rr, "g_\ast", near end, crossing over]
          && \etfdtlgrp(X) \arrow[r]  \arrow[d]
          & \etfdtlgrp(\dM_{g,n+1}) \arrow[d]
          \\
          & \etfdtlgrp(Y) \arrow[r] \arrow[ur,"\ph_\ast"]
          & \etfdtlgrp(S) \arrow[r]
          & \etfdtlgrp(\dM_{g,n}).
    \end{tikzcd}	
\]

The map $\ph_\ast$ appearing in the diagram by assumption \labelcref{propitem:1-truncated homotopy types} is compatible with $f'_{\!\ast}$, and then a diagram chase using the indicated surjections shows the required compatibility with $g_\ast$.
\end{proof}

%%% Local Variables:
%%% mode: LaTeX
%%% TeX-master: "../haupt"
%%% End:

%% file: content/finalize-proof.tex
The ultimate goal of this section is to show \Cref{maintheorem-sharp}, a refined version of \cref{maintheorem}.
We start with some lemmas that lead to a comparison of the notions  $\pi_1$-open versus $\etapioneop$ in \Cref{thm:eta-open}.

\begin{lemma}
\label{lem:Katz-Lang plus epsilon}
Let $S$ be a normal connected qcqs scheme with geometric generic point $\bar \eta \to S$.
Let $f \colon Y \to S$ be of finite type with $Y$ normal and connected, with geometrically connected generic fibre $Y_{\!\bar \eta}$ and assume that the smooth locus of $f$ surjects onto $S$.
Let $y$ be a geometric point of $Y_{\!\bar \eta}$.
Then the sequence 
\[
 \etfdtlgrp(Y_{\!\bar \eta}, y) \lang  \etfdtlgrp(Y, y) \longrightarrow  \etfdtlgrp(S,\bar \eta) \lang 1
\]
of homotopy groups in low degrees is exact.
\end{lemma}

\begin{proof} Using noetherian approximation, we may assume that all occurring schemes are noetherian.
If $Y \to S$ is smooth, then the assertion of \Cref{lem:Katz-Lang plus epsilon} is proven in  \cite[Lemma~2]{KL}. 

Let $Y^0$ be the open in $Y$ where $f$ is smooth.
Since $Y_{\!\bar \eta}$ is unibranch and connected, the geometric fibre $Y^0_{\!\bar \eta}$ is a connected open in $Y_{\!\bar \eta}$, i.e., the restriction $Y^0 \to S$ has a geometrically connected generic fibre.
We may assume that $y$ is contained in $Y^0_{\!\bar \eta}$ and obtain a commutative diagram
\[
\begin{tikzcd}
   \etfdtlgrp(Y^0_{\!\bar \eta}, y)  \rar  \arrow[d, two heads]  
   & \etfdtlgrp(Y^0,y) \rar  \arrow[d, two heads]  
   & \etfdtlgrp(S,\bar \eta) \rar \arrow[d,equal]
   & 1 \\
   \etfdtlgrp(Y_{\!\bar \eta}, y)  \rar 
   & \etfdtlgrp(Y, y) \rar 
   & \etfdtlgrp(S,\bar \eta)  \rar& 1,
\end{tikzcd}
\]
in which the top row is exact by \cite[Lemma~2]{KL}.
Since the middle vertical map is surjective, the bottom row is also exact.
\end{proof}

\begin{lemma}\label{lem:centre-free}
Let $X$ be a smooth, hyperbolic curve over a separably closed field. Then, for any geometric point $x$ of $X$, the profinite group  $\hat\pi_1(X,x)$ is strongly centre-free. 
\end{lemma}

\begin{proof}
See \cite[Proposition~8, Proposition~18]{andersonExactnessPropertiesProfinite1974} in the characteristic $0$ case, and  \cite[\nopp (1.11)]{Tama-groth} in arbitrary characteristics. 
\end{proof}

\begin{proposition}
\label{thm:eta-open}
  Let $S$ be a normal, connected qcqs scheme in characteristic $0$ with generic point $\eta$.
  Let $X \to S$ be a hyperbolic curve and let $Y \to S$ be of finite type with $Y$ normal and connected, with geometrically connected generic fibre and assume that the smooth locus of\/ $Y \to S$ surjects onto~$S$.
 
  Let $\varphi: \ethtpytype(Y) \to \ethtpytype(X)$ be a map over $\ethtpytype(S)$ and let $\varphi_\eta: \ethtpytype(Y_\eta) \to  \ethtpytype(X_\eta)$ be the map over $\ethtpytype(\eta)$ (defined up to homotopy) induced in view of \cref{thm:generalisation-of-Friedlander} from $\varphi$ by base change as in \cref{constr:basechange-for-maps-of-etale-homotopy-types}.

Then $\varphi$ is $\pi_1$-open if and only if $\varphi_\eta$ is $\pi_1$-open.
\end{proposition}

\begin{proof}
That `$\varphi_\eta$ is $\pi_1$-open' implies `$\varphi$ is $\pi_1$-open' follows easily from the commutative diagram of profinite groups with outer homomorphisms
\[
\begin{tikzcd}
  \etfdtlgrp(Y_{\!\eta}, y)\rar{\etfdtlgrp(\varphi_{\eta})} \arrow[d, two heads] & \etfdtlgrp(X_\eta, x)\arrow[d,two heads]\\
  \etfdtlgrp(Y, y) \rar{\etfdtlgrp(\varphi)} & \etfdtlgrp(X, x).
\end{tikzcd}
\] 

Let $\bar{\eta}$ be a geometric point over $\eta$.
The diagram 
\[
\begin{tikzcd}
  1 \rar 
  & \etfdtlgrp(Y_{\!\bar\eta}, y)  \rar \dar{\etfdtlgrp(\varphi_{\bar{\eta},\ast})} 
  & \etfdtlgrp(Y_{\!\eta}, y) \rar\dar{\etfdtlgrp(\varphi_{\eta})} 
  & \etfdtlgrp(\eta,\bar \eta) \rar \arrow[d,equal] 
  & 1 \\
  1 \rar 
  & \etfdtlgrp(X_{\bar\eta}, x)  \rar 
  & \etfdtlgrp(X_\eta, x) \rar 
  & \etfdtlgrp(\eta,\bar \eta) \rar
  & 1
\end{tikzcd}
\]
has exact rows by \cite[\nopp IX,6.1]{SGA1}, and shows that $\etfdtlgrp(\varphi_\eta)$ is open if and only if $\etfdtlgrp(\varphi_{\bar{\eta}})$ is open.

Moreover, we have a short exact sequence
\[
1 \longrightarrow \etfdtlgrp(X_{\bar{\eta}}, x) \stackrel{\alpha}{\longrightarrow} \etfdtlgrp(X, x) \longrightarrow \etfdtlgrp(S,\bar \eta) \longrightarrow 1.
\] 
The exactness except the injectivity of $\alpha$ follows from \Cref{thm:generalisation-of-Friedlander}.  
Moreover, by \cite[Proposition~1.4]{Friedl-elfib}, $\ker(\alpha)$ lies in the centre of $\etfdtlgrp(X_{\bar{\eta}}, x)$ (in fact, this is a formal consequence of $X\to S$ being a quasifibration). But the centre is trivial by \Cref{lem:centre-free}.

Together with \cref{lem:Katz-Lang plus epsilon}, we obtain a diagram with exact rows
\[
\begin{tikzcd}
  & \etfdtlgrp(Y_{\!\bar\eta}, y)  \rar \dar{\etfdtlgrp(\varphi_{\bar{\eta}})} 
  & \etfdtlgrp(Y, y) \rar\dar{\etfdtlgrp(\varphi)} 
  & \etfdtlgrp(S,\bar \eta) \rar \arrow[d,equal] 
  & 1 \\
  1 \rar 
  & \etfdtlgrp(X_{\bar\eta}, x)  \rar 
  & \etfdtlgrp(X, x) \rar 
  & \etfdtlgrp(S,\bar \eta) \rar
  & 1 \ .
\end{tikzcd}
\]
A diagram chase shows that the index of the image of $\etfdtlgrp(\varphi)$ agrees with the index of $\etfdtlgrp(\varphi_{\bar {\eta}})$, hence $\etfdtlgrp(\varphi)$ is open if and only if $\etfdtlgrp(\varphi_{\bar {\eta}})$ is open. 
\end{proof}

\begin{theorem} 
\label{maintheorem-sharp}
Let $S$ be a normal, connected, excellent scheme in characteristic $0$, such that the residue field at the generic point $\eta$ of $S$ is sub-$p$-adic. Let $X \to S$ be a relative hyperbolic curve. 

Then, for all\/  $Y \to S$ flat and of finite type with connected geometric generic fibre, such that $Y$ is normal, connected, and the smooth locus of $Y \to S$ surjects on to $S$, the natural map 
         \begin{equation}
         \label{eq:main map}
                \ethtpytype \from \Hom_{S}^{\dom}(Y, X) 
                \isomto 
                \map^{\pioneop}_{\ethtpytype(S)}\big(\ethtpytype(Y), \ethtpytype(X)\big)
         \end{equation}
is an equivalence of anima.
In particular, $\map^{\pioneop}_{\ethtpytype(S)}\big(\ethtpytype(Y), \ethtpytype(X)\big)$ is discrete.
\end{theorem}

\begin{proof} 
  By now  we have constructed a cartesian diagram
\begin{equation}
\label{eq:main square-copy}
    \begin{tikzcd}
        \Hom_{S}^{\dom}(Y, X) \arrow[r, "{(4)}"] \arrow[d, hook, "{(1)}"']
        & \ethtpycls[Y][X][S]^{\etapioneop} \arrow[d, hook, "{(2)}"] \\ 
        \Hom_{\eta}^{\dom}(Y_{\!\eta}, X_{\eta}) \arrow[r, "{(3)}"', "{\sim}"]
        & \etopcls[Y_{\!\eta}][X_\eta][\eta]
    \end{tikzcd}
\end{equation}
as predicted in \eqref{eq:main square} in the introduction.
The map $\diagr{1}$ is the base change along $\eta \hookrightarrow S$ and injective by \Cref{prop:injectivity-via-generic-fibre}.
The map $\diagr{2}$ is the base change map constructed in \cref{subsec:the-base-change-map} and is injective by \Cref{thm:spreading-out-homotopies} (spreading out homotopies).
The map $\diagr{3}$ is induced by $\ethtpytype$ and is bijective by \Cref{cor:Mochizuki-via-homotopy-types}, the homotopy-theoretic reformulation of Mochizuki's Theorem.
Moreover, the diagram is cartesian by \Cref{prop:extension of curves controlled by homotopy} (spreading out morphisms).
\Cref{thm:eta-open} shows the equality
\[
\ethtpycls[Y][X][S]^{\etapioneop} = \ethtpycls[Y][X][S]^{\pi_1\text{-op}}
\]
in the upper right corner, namely that the two conditions $\etapioneop$ and $\pi_1\text{-op}$ are equivalent.
We therefore have proven that the map
\[
      \ethtpytype \from \Hom_{S}^{\dom}(Y, X) \longrightarrow \etopcls[Y][X][S]
\]
is bijective.
To complete the proof of \Cref{maintheorem-sharp} it remains to show that $\map^{\pioneop}_{\ethtpytype(S)}(\ethtpytype(Y), \ethtpytype(X))$ is a discrete anima.
By \Cref{thm:spreading-out-homotopies}, the base change map
\[
	\map^{\pioneop}_{\ethtpytype(S)}\big(\ethtpytype(Y), \ethtpytype(X)\big) 
	\longrightarrow 
	\map^{\pioneop}_{\ethtpytype(\eta)}\big(\ethtpytype(Y_{\!\eta}), \ethtpytype(X_\eta)\big)
\]
is a monomorphism.
Therefore it suffices to show that the right-hand side is discrete.
By \cref{prop:hyperbolic-curves-are-K-pi-1}, we have a natural equivalence
\[
	\map^{\pioneop}_{\ethtpytype(\eta)}\big(\ethtpytype(Y_{\!\eta}), \ethtpytype(X_\eta)\big) 
	\isomto 
	\map^{\pioneop}_{\B\!\etfdtlgrp(\eta)}\big(\B\!\etfdtlgrp(Y_{\!\eta}), \B\!\etfdtlgrp(X_\eta)\big).
\]
Finally, $\etfdtlgrp(X_{\bar{\eta}})=\ker(\etfdtlgrp(X_{{\eta}}) \lsurj \etfdtlgrp(\eta)) $ is strongly centre-free by \Cref{lem:centre-free}.
We conclude by \Cref{cor:HomAnimadiscrete} that $\map^{\pioneop}_{\B\!\etfdtlgrp(\eta)}(\B\!\etfdtlgrp(Y_{\!\eta}), \B\!\etfdtlgrp(X_\eta))$ is discrete.
\end{proof} 

%%% Local Variables:
%%% mode: LaTeX
%%% TeX-master: "../haupt"
%%% End:

%% file: content/appendix.tex
%----------------------------------------------------------------------------------------------------------------------------------
\section{Some (profinite) homotopy theory}
\label{sec:app}

%----------------------------------------------------------------------------------------------------------------------------------
\subsection{Profinite anima}
\label{app:profinite-anima}

In the following we recall the construction of an $\infty$-categorical refinement of profinite homotopy theory.

\begin{recollection}[Pro-categories of $\infty$-categories] \label{rec:pro-categories}
    Let $\catC$ be an $\infty$-category.
    \begin{thmlist}
        \item \label{recitem:pro-category}
        A functor $j \from \catC \to \Pro(\catC)$ \emph{exhibits $\Pro(\catC)$ as a pro-category of $\catC$} if the following two conditions are satisfied.\smallskip
        \begin{thmlist}
            \item
            The $\infty$-category $\Pro(\catC)$ admits small cofiltered limits.
            \item
            Given any other $\infty$-category $\catE$ with small cofiltered limits, precomposition with $j$ induces an equivalence of $\infty$-categories
                \[
                    j^{*} \from \Fun^{\cofilt}(\Pro(\catC), \catE) \longrightarrow \Fun(\catC, \catE),
                \]
                where $\Fun^{\cofilt}(\Pro(\catC), \catE) \subset \Fun(\Pro(\catC), \catE)$ denotes the full subcategory spanned by the functors that preserve cofiltered limits.
        \end{thmlist}
        \item \label{recitem:prorepresentable}
            A functor $F \from \catC \to \Ani$ is called \defding{prorepresentable} if there exists a cofiltered diagram
                \[
                \catI \longrightarrow \catC,\ i \mapsto c_{i} \quad \text{and an equivalence} \quad F \simeq \limit_{i} j(c_{i}),
                \]
                where $j \from \catC \to \Fun(\catC, \Ani)^{\op},\ c \mapsto \rep^{c} = \map_{\catC}(c, \blank)$ denotes the Yoneda embedding.
                We write $\Fun^{\pro}(\catC, \Ani) \subset \Fun(\catC, \Ani)$ for the full subcategory spanned by the prorepresentable functors.
                Moreover, we follow Deligne and write
                \[
                    \prolimit_{i} c_{i} = \limit_{i} j(c_{i}).
                \]
        \item \label{recitem:lex}
              By the dual of \HTT{}{5.3.5.4}, the functor
              \[
                    \catC \longrightarrow \Fun^{\pro}(\catC, \Ani)^{\op},\quad c \mapsto \rep^{c} = \map_{\catC}(c, \blank)
              \]
              induced by the Yoneda embedding exhibits $\Fun^{\pro}(\catC, \Ani)^{\op}$ as a pro-category of $\catC$.
              Consequently, any object $c \in \Pro(\catC)$ admits a representation $c \simeq \prolimit_{i} c_{i}$ for some cofiltered diagram $\{c_{i}\}_{i \in \catI}$ in $\catC$.
              Moreover, if $\catC$ admits finite limits, then a functor is prorepresentable if and only if it preserves finite limits.
              Therefore, $\Pro(\catC) = \Fun^{\pro}(\catC, \Ani)^{\op}$ coincides with the full subcategory $$\Fun^{\lex}(\catC, \Ani)^{\op} \subset \Fun(\catC, \Ani)^{\op}$$ of functors preserving finite limits (a.k.a.\ left exact functors) in this case.
    \end{thmlist}
\end{recollection}

\begin{remark} 
The above definition of pro-categories of $\infty$-categories extends the corresponding notion for $1$-categories.
Let $\catC$ be a $1$-category and let $j \from \catC \to \Pro(\catC)$ be its $1$-categorical pro-category.
Then the induced functor 
\[
	\nerve(j) \from \nerve(\catC) \to \nerve(\Pro(\catC))
\]
exhibits $\nerve(\Pro(\catC))$ as $\infty$-categorical pro-category of $\nerve(\catC)$ by (the dual of) \HTT{}{5.3.5.6}.
 
 For an $\infty$-category $\catC$, we obtain a functor of $1$-categories
\[
	\h\!\Pro(\catC)\lang \Pro(\nerve(\h\! \catC))=\Pro (\h\!\catC),
\]
which in general is far from being an equivalence.
\end{remark}

\begin{definition}[Profinite anima]
\label{def:profinite-anima}
        Let $B$ be an anima.
        \begin{thmlist}
            \item
	    The anima $B$ is \defding{truncated} if  there exists a number $N$ such that $\pi_i(B,b) = 0$ for all $b \in B$ and all $i \geq N$.
            \item
            The anima  $B$ is said to be \defding{$\pi$-finite} if $B$ is truncated and $\pi_i(B,b)$ is finite for all $b \in B$ and all $i\ge 0$.
            \item
            We write $\SigmafinAni \subset \Ani$ for the full subcategory spanned by the $\pi$-finite anima.
            \item
            We call $\Pro(\SigmafinAni)$ the $\infty$-category of \defding{profinite anima}.
        \end{thmlist}
\end{definition}

\begin{recollection}[homotopy and cohomology groups]
  \label{rec:homotopy-and-cohomology-of-profinite-anima}
  We write $\finSets$ (resp. $\finGrp$, $\finAb$) for the category of \emph{finite sets} (resp. \emph{finite (abelian) groups}), and $\D(\bZ)$ for the derived $\infty$-category of the category $\Ab$ of abelian groups.
  \begin{deflist}
    \item 
    The connected component functor $\conncomp \from \pifinAni \to \finSets$ extends along cofiltered limits to a functor
    \[
      \conncomp \from \Pro(\pifinAni) \longrightarrow \Pro(\finSets), \quad B = \prolimit_i B_i \mapsto \prolimit_i \conncomp(B_i).
    \]
    We refer to $\conncomp (B)$ as the \defding{(profinite) set of connected components of $B$}.

    \item 
    For $k \geq 1$, the homotopy group functor $\htpygrp_k \from \ptpifinAni \to \finGrp$ extends along cofiltered limits to a functor
      \begin{alignat*}{3}
        \htpygrp_k \from \Pro(\pifinAni)_{\terminal} &= \Pro(\ptpifinAni) &&\longrightarrow \Pro(\finGrp) \\
        (B, b) &= \prolimit_i (B_i, b_i) &&\longmapsto \prolimit_i \htpygrp_k(B_i, b_i).
      \end{alignat*}
    We refer to $\htpygrp_{k}(B, b)$ as the \defding{$k$-th (profinite) homotopy group of $(B, b)$}. 
    For $k \geq 2$, the essential image of the functor $\pi_k$ is contained in the full subcategory $\Pro(\finAb) \subset \Pro(\finGrp)$ of profinite abelian groups.
   
    \item 
    Let $M$ be an abelian group.  Since $\D(\bZ)$ admits filtered colimits, the singular cochain complex functor 
    $\C^{*}(\blank, M) \from \pifinAni \to \D(\bZ)^{\op}$ 
    extends to a cofiltered limit preserving functor
    \[
      \C^{*}(\blank, M) \from \Pro(\pifinAni) \longrightarrow \D(\bZ)^{\op}, \quad B = \prolimit_i B_i \mapsto \colimit_i \C^{*}(B_i, M).
    \]
    For any integer  $k \geq 0$, we write $\rH^k(B, M)$ for the abelian group $\rH^k(\C^{*}(B, M))$ and call it the \emph{$k$-th cohomology group of $B$ with coefficients in $M$}. Since filtered colimits  are exact in $\Ab$, we have $\rH^k(B, M)  = \colimit_i \rH^k(B_i, M)$.
  \end{deflist}
  We refer the reader to \cite[\SAGsubsec{E.5.2}, \SAGsubsec{E.7.1}]{SAG} for more details.
\end{recollection}

The following lemma is an immediate consequence of \SAG{}{E.4.6.1} (in the case $Y =\ \point$).

\begin{lemma}
  \label{lem:n-truncated-profinite-anima}
  Let $n \geq -1$.
  The following are equivalent for a profinite anima $B$.
  \begin{thmlist}
    \item $B$ is an $n$-truncated object of $\Pro(\pifinAni)$.    
    \item For all $k > n$ and all $b \in B$, the (profinite) homotopy group $\htpygrp_k(B, b)$ vanishes.
    \item $B$ lies in the essential image of the inclusion $\Pro(\Ani_{\pi, \leq n}) \subset \Pro(\pifinAni)$ induced by the inclusion $\Ani_{\pi, \leq n} \subset \pifinAni$.
  \end{thmlist}
\end{lemma}

\begin{corollary}
  \label{cor:n-truncation-profinite-anima}
  Let $n \geq -1$ and $B$ a profinite anima.
  \begin{thmlist}
    \item \label{coritem:n-truncation} 
    The inclusion $\Pro(\pifinAni)_{\leq n} \subset \Pro(\pifinAni)$ admits 
    a cofiltered limit preserving left adjoint
    \[
      \trunc_{\leq n} \from \Pro(\pifinAni) \longrightarrow \Pro(\pifinAni)_{\leq n},
    \]
    called the \emph{$n$-truncation functor}.
    \item \label{coritem:iso-on-htpy-grp-up-to-n} 
    The unit map $B \to \trunc_{\leq n}\! B$ induces isomorphisms $\htpygrp_k(B, b) \isomto \htpygrp_k(\trunc_{\leq n}\! B, b)$ for all $k \leq n$ and $b \in B$.
  \end{thmlist}
\end{corollary}

\begin{proof}
  By \HTT{}{5.5.6.18} and \HTT{}{5.2.8.16}, the inclusion $\Ani_{\leq n} \subset \Ani$ admits a left adjoint $\trunc_{\leq n}$ satisfying \labelcref{coritem:iso-on-htpy-grp-up-to-n}, and therefore the functor $\trunc_{\leq n}$ restricts to a left adjoint of $\Ani_{\pi, \leq n} \subset \Ani_{\pi}$.
  Thus, by extending along cofiltered limits, we obtain an adjunction as in \labelcref{coritem:n-truncation} but for the subcategory $\Pro(\Ani_{\pi, \leq n})$.
  Since homotopy groups of profinite anima are formed levelwise, property \labelcref{coritem:iso-on-htpy-grp-up-to-n} remains true.
  Since, by \Cref{lem:n-truncated-profinite-anima}, the essential image of $\Pro(\Ani_{\pi, \leq n}) \subset \Pro(\pifinAni)$ is precisely given by $\Pro(\pifinAni)_{\leq n}$, we conclude.
\end{proof}

\begin{recollection}[Profinite completion]
\label{rec:profinite-completion}
        \begin{thmlist}
            \item The inclusion $\SigmafinAni \inj \Ani$ induces a fully faithful embedding
                    \[
                         \Pro(\SigmafinAni) \linj \Pro(\Ani) .
                    \]
            \item Since $\SigmafinAni \subset \Ani$ is stable under finite limits, $\Pro(\SigmafinAni) \subset \Pro(\Ani)$ admits a left adjoint
                    \[
                        (-)^\wedge_\pi \from \Pro(\Ani) \longrightarrow \Pro(\SigmafinAni),
                    \]
                    called \defding{profinite completion}.
        \end{thmlist}
\end{recollection}

An important feature of profinite anima, which does not hold for arbitrary pro-anima, is that equivalences can be detected on homotopy groups.

\begin{theorem}[Profinite Whitehead theorem, {\SAG{}{E.3.1.6}}]
  \label{thm:profinite-whitehead-theorem}
  The following are equivalent for a map $\varphi \from E' \to E$ of profinite anima.
  \begin{thmlist}
    \item The map $\varphi$ is an equivalence of profinite anima.
    \item The map $\varphi$ induces an isomorphism on homotopy groups, more precisely:
    \begin{thmlist}
      \item The induced map $\varphi_{*} \from \htpygrp_0(E') \to \htpygrp_0(E)$ is an isomorphism of profinite sets, and
      \item for all $e' \in E'$ with image $e = \varphi(e') \in E$, and all $n \geq 1$, the induced homomorphism
      \[
        \varphi_{*} \from \htpygrp_n(E', e') \to \htpygrp_n(E, e)
      \]
      is an isomorphism of profinite groups.
    \end{thmlist}
  \end{thmlist}
\end{theorem}

Moreover, we will make use of the following profinite version of the Hurewicz theorem.

\begin{proposition}[Profinite Hurewicz theorem, {\SAG{}{E.7.4.1}}]
  \label{lem:profinite-hurewicz-theorem}
  Let $E$ be a profinite anima which is $n$-connective for some $n \geq 1$, let $e$ be a point of $E$, and let $M$ be an abelian group.
  There are canonical isomorphisms
  \[
    \rH^{k}(E, M) = 
    \begin{cases}
      M & \text{if\/ } k = 0, \\
      \,0 & \text{if\/ } 0 < k < n, \text{and} \\
      \Hom(\htpygrp_n(E, e), M) & \text{if\/ } k = n,
    \end{cases}
  \]
 where $M$ is considered as a discrete topological group and homomorphisms are continuous. 
\end{proposition}

%----------------------------------------------------------------------------------------------------------------------------------
\subsection{Cartesian squares}

The following lemma characterises cartesian squares of  anima and of profinite anima in terms of fibres.

\begin{lemma}[fibrewise criterion of pullbacks]
    \label{lem:criterion-homotopy-pullback-fibres}
    The following are equivalent for a commutative square
    \[
        \begin{tikzcd}
            E' \arrow[r] \arrow[d]
              & E \arrow[d] \\
            B' \arrow[r]
              & B
        \end{tikzcd}
    \]
    of  anima or of profinite anima, respectively:
    \begin{thmlist}
        \item
        \label{lemitem:fibre product of anima}
        The square is cartesian.
        \item
        \label{lemitem:fibre of canonical map}
        For every point $p \in  E \times_{B} B'$, the fibre $\fib_p(E' \to E \times_{B} B')$ is contractible.
        \item
        \label{lemitem:homotopy equivalent fibres}
        For every point $b' \in B'$ with image $b \in B$, the induced map
            \[
                \fib_{b'}(E'\to B') \longrightarrow \fib_{b}(E\to B)
            \]
            is an equivalence.
    \end{thmlist}
\end{lemma}

\begin{proof}
  If \labelcref{lemitem:fibre of canonical map} holds, then the homotopy sequence for $E' \to E \times_{B} B'$ and every point $e' \in E'$ with image $p$ in $E \times_{B} B'$ yields isomorphisms $\pi_i(E',e') \to \pi_i(E \times_{B} B',p)$ and all $i \in \bN$, so that \labelcref{lemitem:fibre product of anima} follows from \Cref{thm:profinite-whitehead-theorem}.
  The implication \labelcref{lemitem:fibre product of anima} $\Rightarrow$ \labelcref{lemitem:homotopy equivalent fibres} is formal.

For the remaining direction \labelcref{lemitem:homotopy equivalent fibres} $\Rightarrow$
\labelcref{lemitem:fibre of canonical map}, we consider an arbitrary point $p \in E \times_{B} B'$ and its images $e\in E$, $b' \in B'$ and $b \in B$. A diagram chase in
\[
	\begin{tikzcd}
        \fib_{b'}(E'\to B')  
        \cartesian \arrow[d] \arrow[r]
        & E' \arrow[d]
        \\
	\fib_{b}(E\to B) \cartesian \arrow[d] \arrow[r]
        & E \times_{B} B' \cartesian \arrow[d] \arrow[r]
        & E \arrow[d]
        \\
        \point \arrow[r, "b'"]
        & B' \arrow[r] & B   .
     	\end{tikzcd}
\]
shows that $\fib_p(E'\to E\times_B B')$ is equivalent to the homotopy fibre of $\fib_{b'}(E'\to B') \to \fib_{b}(E\to B)$ in the point induced by $e$. Since the homotopy fibre of an equivalence is contractible, this proves the claim.
\end{proof}

\begin{corollary} [pasting and cancellation of pullbacks]\label{cor:cancellation of pullbacks}
Let 
\begin{equation}\label{2out3squares}
\begin{tikzcd}
E''\rar\dar{}&E'\rar\dar{}&E\dar{}\\
B''\rar&B'\rar&B
\end{tikzcd}
\end{equation}
be a commutative diagram of anima or of profinite anima, respectively. 
\begin{enumerate}[(1)]
  \item \label{coritem:pullback-pasting}
  If the right square is cartesian, then the left square is cartesian if and only if the outer rectangle is cartesian.
  \item \label{coritem:pullback-cancellation}
  If $\pi_0(B'') \to \pi_0(B')$ is surjective and the left square is cartesian, then the right square is cartesian if and only if the outer rectangle is cartesian.
\end{enumerate}
\end{corollary}

\begin{proof} Assertion~\ref{coritem:pullback-pasting} follows from formal properties of cartesian squares, see (the dual of) \HTT{}{4.4.2.1}. 
In~\ref{coritem:pullback-cancellation} it suffices to show that the right square is cartesian if the outer rectangle is.

For this, using the criterion of \Cref{lem:criterion-homotopy-pullback-fibres}, it suffices to show that for every point $b' \in B'$ with image $b \in B$, the induced map $\fib_{b'}(E'\to B') \to \fib_{b}(E'\to B')$ is an equivalence.
Since $\pi_0(B'') \to \pi_0(B')$ is surjective, we can assume that $b'$ is the image of  point $b''$ of $B''$.
We obtain maps
\[
\fib_{b''}(E''\to B'')\to \fib_{b'}(E'\to B')\to \fib_{b}(E\to B).
\]
Since the left and the outer rectangle of \ref{2out3squares} are cartesian, the maps $\fib_{b''}(E''\to B'')\to\fib_{b'}(E'\to B')$ and $\fib_{b''}(E''\to B'')\to\fib_{b}(E\to B)$ are equivalences. Hence also  $\fib_{b'}(E'\to B')\to\fib_{b}(E\to B)$ is a equivalence, as required. 
\end{proof}

%----------------------------------------------------------------------------------------------------------------------------------
\subsection{Monomorphisms}
\label{subsec:monomorphisms}

\begin{recollection}
  \label{rec:monomorphism-and-epimorphism}
  Let $\catC$ be an $\infty$-category and let $z$ be an object of $\catC$.
  \begin{thmlist}
      \item A map $f \from B' \to B$ of anima is said to be a \defding{monomorphism} if it induces an injection on connected components and an isomorphism on all homotopy groups with respect to all base points.
      \item A map $f \from y \to x$ in $\catC$ is said to be a \defding{monomorphism with respect to $z$} (resp.\ \defding{monomorphism}) if the map
          \[
              f_{\!*} \from \map_{\catC}(z, y) \longrightarrow \map_{\catC}(z, x), \quad g \mapsto f \circ g
          \]
          is a monomorphism of anima (resp.\ for every $z \in \catC$).
          Note that there is no ambiguity in the case $\catC=\Ani$, see \HTT{}{5.5.6.9}. 
          This property is also called \defding{$(-1)$-truncated}. 
      \item A map $f \from y \to x$ in $\catC$ is said to be an \defding{epimorphism with respect to $z$} (resp.\ \defding{epimorphism}) if the induced map
          \[
              f^{*} \from \map_{\catC}(x, z) \longrightarrow \map_{\catC}(y, z), \quad g \mapsto g \circ f
          \]
          is a monomorphism (resp.\ for every $z \in \catC$).
  \end{thmlist}
\end{recollection}

\begin{remarks}
  \label{rem:monomorphisms}
  Let $\catC$ be an $\infty$-category with finite limits and let $\catI$ be a simplicial set.
  \begin{thmlist}
      \item \label{remitem:mono-via-diagonal}
      	A map $f \from y \to x$ in $\catC$ is a monomorphism if and only if the diagonal
	$\Delta_{f} \from y \to y \times_{x} y$ is an equivalence by \HTT{}{5.5.6.15}.

      \item \label{remitem:monos-in-functor-category}
      	As limits and equivalences in $\Fun(\catI, \catC)$ are detected pointwise,
	the preceding remark implies that a natural transformation $\alpha \from F \Rightarrow G$
	between two functors $F, G \from \catI \to \catC$ is a monomorphism if and only if it is
	componentwise a monomorphism.

      \item \label{remitem:left-exact-functor-preserves-monos}
      	Any left exact functor between $\infty$-categories admitting finite limits
	preserves monomorphisms by \HTT{}{5.5.6.16}.

      \item \label{remitem:monos-are-stable-under-limits}
      	In particular, since $\lim_{i \in \catI} \from \Fun(\catI, \catC) \to \catC$ is left exact, it
	preserves monomorphisms, i.e., monomorphisms are stable under limits.

      \item \label{remitem:dualizing-for-epis}
      	Dualising all of the above yields analogous statements for epimorphisms.

      \item \label{remitem:filtered-colimits-in-anima}
      Since filtered colimits in $\Ani$ commute with finite limits, it follows from~\ref{remitem:left-exact-functor-preserves-monos} that monomorphisms of anima are stable under filtered colimits.
  \end{thmlist}
\end{remarks}

\begin{lemma}
\label{lem:cancellation-of-monos}
Let $\catC$ be an $\infty$-category, and let $g \from z \to y$ and $f \from y \to x$ be morphisms in $\catC$. If $f$ is a monomorphism, then $g$ is a monomorphism if and only if $f \circ g$ is a monomorphism.
\end{lemma}

\begin{proof} Since being a monomorphism can be tested on $\map_{\catC}(z,-)$, it suffices to treat the case $\catC = \Ani$.
  The implication for the composition being obvious, we now assume that $f \circ g$ is  a monomorphism.
  In this case, note that the injectivity of $\htpygrp_{0}(f \circ g)$ implies that $\htpygrp_{0}(g)$ is injective.
  Given $m \geq 1$, then by assumption $\htpygrp_{m}(f \circ g)$ and $\htpygrp_{m}(f)$ are isomorphisms, hence so is $\htpygrp_{m}(g)$.
\end{proof}

\begin{lemma}
\label{lem:right-adjoint-faithful}
Let $\catC$ and $\catD$ be $\infty$-categories, and $G \from \catD \to \catC$ a functor with left adjoint $F$.
\begin{thmlist}
  \item The following are equivalent for objects $d, d' \in \catD$.
  \begin{thmlist}
    \item The induced map
      \[
          \begin{tikzcd}
            \map_{\catD}(d, d') \arrow[r, "{G}"] & \map_{\catC}(Gd, Gd')
          \end{tikzcd}
      \]
      is a monomorphism.
    \item The counit $\counit_{d} \from FGd \to d$ is an epimorphism with respect to $d'$.
  \end{thmlist}
  \item The following are equivalent:
    \begin{thmlist}
      \item $G$ is fully faithful.
      \item The counit $\counit \from F \circ G \to \id_{\catD}$ is an equivalence.
    \end{thmlist}
\end{thmlist}
\end{lemma}

\begin{proof}
Given any $d, d'$ in $\catD$, the adjunction supplies a commutative diagram
\[
  \begin{tikzcd}
    \map_{\catD}(d, d') \arrow[rd, "{\blank \circ \counit_{d}}"'] \arrow[rr, "{G}"] & & \map_{\catC}(Gd, Gd') \\
    & \map_{\catD}(FGd, d') \arrow[ru, "{\simeq}"'],
  \end{tikzcd}
\]
in which the right diagonal map is an equivalence.
From this we see:
\begin{thmlist}
  \item $G$ is a monomorphism if and only if $\blank \circ \counit_{d}$ is.
  \item $G$ is fully faithful if and only if $\blank \circ \counit_{d}$ is an equivalence for every $d, d'$.
        By varying $d'$, we see that this in particular means that $\counit_{d}$ is itself an equivalence by the Yoneda lemma.
        Since natural equivalences are detected pointwise, we conclude. \qedhere
\end{thmlist}
\end{proof}

\begin{lemma}
  \label{lem:epimorphisms-in-over-category}
  Let $\catC$ be an $\infty$-category with pushouts, $s \in \catC$ some object and $p \from \overcat{\catC}{s} \to \catC$ the forgetful functor.
  Then the following are equivalent for a map $f \from y \to x$ in $\overcat{\catC}{s}$.
  \begin{thmlist}
      \item $f$ is an epimorphism in $\overcat{\catC}{s}$.
      \item $p(f)$ is an epimorphism in $\catC$.
  \end{thmlist}
\end{lemma}

\begin{proof}
The map  $f$ is an epimorphism if and only if the codiagonal $\nabla_{\! f} \from y \amalg_{x} y \to y$ is an equivalence in $\overcat{\catC}{s}$ and similarly for $p(f)$ in $\catC$, see \Cref{rem:monomorphisms}.
  Since, by \kerodon{02KB}, the forgetful functor $p \from \overcat{\catC}{s} \to \catC$ creates colimits and, by combining \kerodon{019K} with \kerodon{0236}, is furthermore seen to be conservative, we conclude.
\end{proof}

%----------------------------------------------------------------------------------------------------------------------------------
\subsection{Classifying Anima} \label{subsec:BG}

\begin{recollection}
  \label{rec:BG} 
    \begin{thmlist}
    \item 
    The restriction of the fundamental group functor $\htpygrp_1 \colon \Ani_{\point} \to \Grp$ to the full subcategory $\Ani_{\point}^{\connected} \subset \Ani_{\point}$ of connected pointed anima admits a fully faithful right adjoint
    \[
       \B \from \Grp \linj \Ani_{\point}^{\connected}
    \]
    with essential image given by the pointed, connected and $1$-truncated anima, that carries a group $G$ to its \emph{classifying anima} $\B\!G = \K(G, 1)$, see \HTT{}{7.2.2.12}. 

    \item 
    By restricting the above adjunction to 
    $\pi_1 \colon \Ani_{\htpygrp, \point}^{\connected} \to \finGrp$, and $\B \colon  \finGrp \to \Ani_{\htpygrp, \point}^{\connected}$, 
    and extending along cofiltered limits, we see that also
    \[
      \htpygrp_{1} \from \Pro(\pifinAni)_{\point}^{\connected} \simeq \Pro(\Ani_{\htpygrp, \point}^{\connected}) \to \Pro(\finGrp)
    \]
    admits a fully faithful right adjoint
    \[
      \B \from \Pro(\finGrp) \linj \Pro(\pifinAni)_{\point}^{\connected} .
    \]
    The essential image of $\B$ is given by the pointed, connected, and $1$-truncated profinite anima.
    The \emph{profinite classifying anima} of a profinite group $G$ is the resulting pointed profinite anima~$\B\!G$. 
    In particular, for any connected and pointed profinite anima $(E, e)$, the functor $\B$ induces an isomorphism
    \[
       \Hom_{\Pro(\finGrp)}\big(\htpygrp_1(E, e), G\big) \isomto \map_{\point}\big((E,e), \B\!G\big),
    \]
    hence  $\map_{\point}((E,e), \B\!G)$ is discrete. 
    The functor $\B$ being fully faithful  implies that the counit of the adjunction
    \[
    \pi_1(\B\!G) \isomto G
    \]
    is an isomorphism. 
  \end{thmlist}
\end{recollection}

\begin{recollection}[anima with group action]
  \label{rec:anima-with-G-action}
  Let $G$ be a profinite group considered as a (discrete) group object in $\pfAni$. 
  \begin{thmlist}
  \item\label{recitem:symmetric-monoidal} The product $(E, E') \mapsto E \times E'$ determines a symmetric monoidal structure on $\pfAni$, that we denote by $\pfAni^{\times}$.
    \item With respect to this monoidal structure, $G$ is an \emph{associative algebra} in $\pfAni^{\times}$.
          We write $\pfAni(G) = \RMod_G(\pfAni^{\times})$ for the $\infty$-category of right $G$-modules in $\pfAni^{\times}$ and refer to it as the \emph{$\infty$-category of profinite anima with $G$-action}, see \cite[\SAGsubseclink{E.6.5}]{SAG} and \cite[\HAthmlink{4.2.1.13}, \HAthmlink{4.3.2.16}]{HA}.
       \item\label{recitem:adjunction} Let $\varphi: G \to H$ be a homomorphism between profinite groups. Precomposition with $\varphi$ induces the functor
          \[
            \varphi^{\ast} \from \pfAni(H) \longrightarrow \pfAni(G)
          \]
          that, by \HA{}{4.6.2.17}, admits a left adjoint $\varphi_{\ast}$ computed by the relative tensor product $\blank \otimes_G H$.
          Both constructions are compatible with composition, i.e., $(\psi \circ \varphi)^{\ast} \simeq \varphi^{\ast} \circ \psi^{\ast}$ and $(\psi \circ \varphi)_{\ast} \simeq \psi_{\ast} \circ \varphi_{\ast}$, see \HA{}{4.4.3.14}.

    \item\label{recitem:homotopy-quotient} In the case that $H = 1$, we have $\pfAni(1) = \pfAni$ and     
    write
      \[
          \blank \modmod G
          = \blank \otimes_G 1 \from \pfAni(G) \longrightarrow \pfAni .
      \]
      We refer to $E\modmod G$ as the \emph{homotopy quotient} of $E$ by $G$.   
  \end{thmlist}
\end{recollection}

\begin{recollection}[pointed versus unpointed]
  \label{rec:pointed-vs-unpointed}
 Let $(E, e)$ and $(B, b)$ be pointed anima with $B$ connected.
\begin{thmlist}
    \item 
    We have a fibre sequence
    \[
      \begin{tikzcd}
        \map_{\point}\big((E,e), (B,b)\big) \arrow[r] \arrow[d] \cartesian & \map(E, B) \arrow[d, "{\ev_e}"] \\
        \point \arrow[r, "b"'] & B,
      \end{tikzcd}
    \]
    where $\ev_e$ is evaluation at $e$ and  the upper horizontal map is the one induced by the forgetful functor $\Ani_{\point} \to \Ani$, see \HTT{}{5.5.5.12}.
        \item 
    Let $\Sup^1=\Delta^1/(0\sim 1)$ be the $1$-sphere in ${\Ani}_*$.  The anima $\LoopAni_bB = \map_{\point}(\Sup^1, (B,b))$ of \emph{loops in~$B$ based at~$b$} 
        is a group object in $\Ani$.
    Since $\pi_1(\Sup^{1}) = \bZ$, for any group $G$,  the fundamental group functor induces an equivalence
    \[
      \Omega_{\point}\!\B\!G = \map_{\point}(\Sup^1, \B\!G) \xrightarrow{\htpygrp_1} \Hom(\bZ, G) = G,
    \]
    where $G \in \sets \subset \Ani$ is discrete.
    \item \label{homquot}
    Let $G$ be a group with classifying anima  $\B\!G$.
    There is an adjoint equivalence of $\infty$-categories
    \[
      \begin{tikzcd}
        \Ani(G) \arrow[r, shift left = 0.2em, "{\blank \modmod G}"] \arrow[r, shift right = 0.2 em, leftarrow, "{\fib_{*}}"']& \overcat{\Ani}{\B\!G},
      \end{tikzcd}
    \]
    where $\Ani(G) = \RMod_{G}(\Ani^{\times})$ denotes the $\infty$-category of anima equipped with an action by the group $G$, and $\blank\modmod G = (\blank \otimes_G \point)$ denotes the homotopy quotient.
    The functor $\fib_*$ carries a map $\varphi \from E \to \B\!G$ to its fibre $\fib_*(\varphi)$ equipped with the natural action of $G$ via the identification $G=\Omega_*\!\B\!G$, see \cite[\HAthmlink{5.2.6.28}, \HAthmlink{5.2.6.29}]{HA}. 
    
    \item
    We can identify the action of $G=\LoopAni_{\point}\!\B\!G$ on the discrete anima $\map_{\point}((E,e), \B\!G)$ with the conjugation action of~$G$ on the set $\Hom(\htpygrp_1(E, e), G)$. 
  \end{thmlist}
\end{recollection}

Let $G = \prolimit_{\alpha} G_{\alpha}$ be a profinite group. We denote by  $\vertbr{\B\!G} = \B\vertbr{G}$ the underlying limit of $\B\!G$ in anima, i.e., the classifying anima of the underlying discrete group $\vertbr{G} = \limit_{\alpha} G_{\alpha} \in \Grp$ of $G$.
We have
\[
\vertbr{\B\!G}= \lim_\alpha \B\!G_\alpha= \map(*,\B\!G).
\]

\begin{proposition}
  \label{lem:concrete-description-of-mapping-anima-into-BG}
  Let $G = \prolimit_{\alpha} G_{\alpha}$ be a profinite group and let  $(E,e) = \prolimit_{\beta} (E_{\beta},e_\beta)$ be a connected pointed profinite anima. Then we have an equivalence of fibre sequences in $\Ani$:
    \[
      \begin{tikzcd}
        \map_{\point}\big((E,e), \B\!G\big) \arrow[d] \arrow[r,"{\htpygrp_1}", "{\sim}"'] & \Hom(\htpygrp_1(E, e), G) \arrow[d] \\
        \map(E, \B\!G) \arrow[d, "{\ev_{e}}"'] \arrow[r, "{\htpygrp_1\!\modmod G}", "{\sim}"'] & \displaystyle\bigsqcup_{\sqrbr{\varphi} \in \Hom(\htpygrp_1(E, e), G)_{G}} \hspace{-1.1cm} \B\vertbr{\Stab_G(\varphi)} \arrow[d] \\
        \B\vertbr{G} \arrow[r, equals] & \B\vertbr{G}.
      \end{tikzcd}
    \]
\end{proposition}

\begin{proof} Given a group $G$ with subgroup $H$, we have a natural identification
\[
(G/H)\modmod G =\ \point\!\modmod H = \B\!H.
\]
Let $M$ be a $G$-set, and fix a set of representatives $m \in [m] \in M/G$. Then the isomorphism of $G$-sets
\[
\bigsqcup_{\sqrbr{m} \in M/G} G/\Stab_G(m) \isomto M
    \]
induces on homotopy quotients mod $G$ an equivalence
\[
      \bigsqcup_{\sqrbr{m} \in M/G} \B\!\Stab_G(m) \isomto M\!\modmod G
\]
    of anima over $\B\!G$. 
    In particular, we have isomorphisms
    \[
      \Hom(\pi_1(E,e), G) \modmod G \isomfrom \hspace{-.7cm}\bigsqcup_{\sqrbr{\varphi} \in \Hom(\htpygrp_1(E, e), G)_G} \hspace{-1cm} \B\!\Stab_G(\varphi).
    \]
Applying the above to the groups $G_\alpha$ occurring in $G = \prolimit_{\alpha} G_{\alpha}$,  we obtain fibre sequences
\[
      \begin{tikzcd}
        \Hom(\htpygrp_1(E, e), G_\alpha) \arrow[r] \arrow[d] \cartesian & \bigsqcup_{\sqrbr{\varphi} \in \Hom(\htpygrp_1(E, e), G_\alpha)_{G_\alpha}} \B\!\Stab_{G_\alpha}(\varphi) \arrow[d] \\
        \point \arrow[r] & \B\!G_\alpha.
      \end{tikzcd}
    \]
Note that the action of $\vertbr{G} = \limit_{\alpha} G_{\alpha}$ on $\Hom(\htpygrp_1(E, e), G) = \limit_{\alpha} \Hom(\htpygrp_1(E, e), G_{\alpha})$ is constructed levelwise. By passing to the limit over $\alpha$ of the above fibre sequence, we therefore obtain the right vertical fibre sequence asserted in the lemma. The left vertical fibre sequence is obtained in the same fashion from the fibre sequences 
\[
      \begin{tikzcd}
        \map_{\point}\big((E,e), \B\!G_\alpha\big) \arrow[r] \arrow[d] \cartesian 
        & \map(E, \B\!G_\alpha) \arrow[d] \\
        \point \arrow[r] & \B\!G_\alpha.
      \end{tikzcd}
    \]
Both compare as claimed. Since the map $\pi_1$ of fibres is an equivalence and the same hold by trivial reasons for the identity of $\B\vertbr{G}$, also the middle map $\pi_1\modmod G$ is an equivalence. 
\end{proof}

Next, we collect some consequences.

\begin{corollary}
    \label{cor:surjection-of-groups-induces-epimorphism-of-groupoids}
    Let $\rho \from G \lsurj H$ be a surjective homomorphism between profinite groups.
    Then the induced map $\B{\!\rho} \from \B{\!G} \to \B{\!H}$ is an epimorphism in $\SigmapfAni[\pi, \leq 1]$.
\end{corollary}

\begin{proof}
 We have to show that, for any profinite groupoid  $\Gamma \in \SigmapfAni[\pi, \leq 1]$, the map
    \[
        \blank \circ \B\!{\rho} \colon \map(\B{\!H}, \Gamma) \longrightarrow  \map(\B{\!G}, \Gamma)
    \]
    is a monomorphism of anima.
    We can write $\Gamma = \prolimit_{\alpha} \Gamma_{\alpha}$ for a suitable choice of $\Gamma_{\alpha} \in \SigmaAni[\pi, \leq 1]$.
    Since, by \Cref{rem:monomorphisms}  \labelcref{remitem:monos-are-stable-under-limits} monomorphisms are stable under filtered colimits, we reduce to $\Gamma \in \SigmaAni[\pi, \leq 1]$.
    Now $\Gamma$ has finitely many connected components.
    As both $\B\!G$ as well as $\B\!H$ are connected, the mapping anima decompose accordingly.
    Using that monomorphisms of anima are stable under finite coproducts, we may reduce to the case that $\Gamma = \B\!K$ for a finite group~$K$.
    In this case, the claim immediately follows from \Cref{lem:concrete-description-of-mapping-anima-into-BG} in the case of a finite group by noting that the surjectivity of $\rho$ implies that
    \[
      \rho^{*} \from \Hom(H, K)_K \linj \Hom(G, K)_{K} 
    \]
    is injective and $\Stab_K(\varphi) = \Stab_K(\varphi \circ \rho)$ for any $\varphi \from H \to K$.
  \end{proof}

\begin{corollary}
  \label{cor:htpy-classes-of-maps-over-BG}
Let $\rho \from G \to H$ be a surjective homomorphism between profinite groups with kernel $N$, and let $(E,e$) be a pointed connected profinite anima over $\B\!H$.
Then we have an equivalence of fibre sequences in $\Ani$
    \[
      \begin{tikzcd}[row sep=2ex]
        \map_{\B\!H}(E, \B\!G) \arrow[d] \arrow[r, "\pi_1\modmod N", "{\sim}"'] & \displaystyle\bigsqcup_{\sqrbr{\varphi}_N \in \Hom_H(\htpygrp_1(E, e), G)_{N}} \hspace{-1.1cm} \B\vertbr{\Stab_N(\varphi)} \arrow[d] \\
        \map(E, \B\!G) \arrow[d, "\rho_*"'] \arrow[r, "{\htpygrp_1\!\modmod G}", "{\sim}"'] & \displaystyle\bigsqcup_{\sqrbr{\varphi} \in \Hom(\htpygrp_1(E, e), G)_{G}} \hspace{-1.1cm} \B\vertbr{\Stab_G(\varphi)} \arrow[d] \\
        \map(E, \B\!H) \arrow[r, "{\htpygrp_1\!\modmod H}", "{\sim}"'] & \displaystyle\bigsqcup_{\sqrbr{\psi} \in \Hom(\htpygrp_1(E, e), H)_{H}} \hspace{-1.1cm} \B\vertbr{\Stab_H(\psi)} ,
      \end{tikzcd}
    \]
    where the fibres are taken with respect to the structure map of $E$ to $\B\!H$ and its image under $\pi_1 \modmod H$, respectively. 
  In particular, $\pi_1\modmod N$ induces a bijection of sets of connected components
  \[
    \htpycls[E][\B\!G][\B\!H] \isomto \Hom_H(\htpygrp_1(E, e), G)_{N}.
  \]
\end{corollary}

\begin{proof} Since a surjective homomorphism between profinite groups can be represented by an inverse system of surjections between finite groups, we can assume that $G$ and $H$ are finite. The horizontal maps  $\htpygrp_1\!\modmod G$ and $\htpygrp_1\!\modmod H$ are equivalences by \Cref{lem:concrete-description-of-mapping-anima-into-BG}. Hence, it remains to check that the vertical sequences are fibre sequences.
The sequence on the left hand side is a fibre sequence by the dual of \HTT{}{5.5.5.12}.
On the right hand side, denoting the structure map of $E$ by $\eta \colon E\to \B\!H$, we have to check  that
  \[
    \begin{tikzcd}
      \bigsqcup_{\sqrbr{\varphi}_N \in \Hom_{H}(\htpygrp_1(E, e), G)_N} \B\!\Stab_N(\varphi) \arrow[r] \arrow[d] & \bigsqcup_{\sqrbr{\varphi}_{G} \in \Hom(\htpygrp_1(E, e), G)_{G}} \B\!\Stab_{G}(\varphi) \arrow[d, "{\rho \circ \blank}"] \\
      \point \arrow[r, "\eta"] & \bigsqcup_{\sqrbr{\psi}_{H} \in \Hom(\htpygrp_1(E, e), H)_{H}} \B\!\Stab_{H}(\psi) 
    \end{tikzcd}
  \]
  is a fibre sequence.
  Since fibres only depend on the connected component of the base, we therefore need to check that
  \[
    \begin{tikzcd}
      \bigsqcup_{\sqrbr{\varphi}_N \in \Hom_{H}(\htpygrp_1(E, e), G)_N} \B\!\Stab_N(\varphi) \arrow[r] \arrow[d] & \bigsqcup_{\sqrbr{\varphi}_{G}, \sqrbr{\rho \circ \varphi}_{H} = \sqrbr{\eta}_{H}} \B\!\Stab_{G}(\varphi) \arrow[d, "{\rho \circ \blank}"] \\
      \point \arrow[r,"\eta"] & \B\!\Stab_{H}(\eta) 
    \end{tikzcd}
  \]
  is a fibre sequence. This can be read off from the long exact sequence of homotopy groups.
\end{proof}

Recall that a profinite group $N$ is \emph{strongly centre-free}  if the centre of every open subgroup of $N$ is trivial.

\begin{corollary}\label{cor:HomAnimadiscrete}
Let $\rho \from G \to H$ be a surjective homomorphism between profinite groups with ker\-nel~$N$, and let $(E,e$) be a pointed connected profinite anima over $\B\!H$.
If $N$ is strongly centre-free, then the subanima
\[
\map_{\B\!H}(E, \B\!G)^{\pioneop} \subset \map_{\B\!H}(E, \B\!G)
\]
of $\pi_1$-open maps is discrete.
\end{corollary}

\begin{proof} The  subanima $\map_{\B\!H}(E, \B\!G)^{\pioneop}$  corresponds under the equivalence $\pi_1\modmod N$ of \Cref{cor:htpy-classes-of-maps-over-BG} to the disjoint union of components indexed by open (conjugacy classes of) homomorphisms $\sqrbr{\varphi}_N$ in $\Hom_H(\htpygrp_1(E, e), G)_{N}$.
Therefore we have to show that $\Stab_N(\varphi)$ is the trivial group if $\varphi \in \Hom_H(\htpygrp_1(E, e), G)$ is open. Indeed, let $U=\mathrm{im}(\varphi)\subset G$ be the open image of $\varphi$ and $n\in \Stab_N(\varphi)$. Then $nun^{-1}=u$ for all $u\in U$. Hence $n$ lies in the centre of the open subgroup of $N$ generated by $n$ and $U\cap N$, but this centre is trivial by assumption.
\end{proof}

%----------------------------------------------------------------------------------------------------------------------------------
\subsection{Simplicial Objects and Resolutions}
\label{subsec:simplicial-objects-and-resolutions}

\begin{recollection}
  \label{rec:localisations}
  Let $L \from \catC \to \catD$ be a functor of $\infty$-categories.
  \begin{thmlist}
    \item The functor $L$ is a \emph{reflective localisation} if it admits a fully faithful right-adjoint $R \from \catD \to \catC$.
    \item If $\catC$ admits pullbacks, we say that the reflective localisation $L$ is \emph{locally cartesian} if for any cospan $d' \to d \ot c$ with $d', d \in \catD$ the canonical map
    \[
      L(d' \times_d c) \longrightarrow d' \times_{d} L(c)
    \]
    is an equivalence.
    See \cite[\S 1.2]{gepner2017} and \cite[\S 3.2]{hoyois2017a}.
  \end{thmlist}
\end{recollection}

Recall that an $\infty$-category $\catC$ is \emph{weakly contractible} if $\catC \to \simplex^{\!0}$ is a weak homotopy equivalence of simplicial sets, see \kerodon{04GW}.

\begin{proposition}
  \label{prop:weakly-contractible-implies-fully-faithful-into-functor-category}
  Let $\catI$ be a weakly contractible $\infty$-category (e.g., $\catI = \simplex^{\op}$) and $\catC$ an $\infty$-category.
  \begin{enumerate}[(1)]
  \item\label{propitem:constant-functor-ff} 
  The constant functor
  \[
    \underline{(\blank)} \from \catC \longrightarrow \Fun(\catI, \catC)
  \]
  is fully faithful.
  \item \label{propitem:colimit-reflective-localisation}
  Assume that $\catC$ admits $\catI$-shaped colimits.
  \begin{enumerate}[(a)]
    \item \label{propsubitem:reflective}
    The functor $\colimit_{\catI} \from \Fun(\catI, \catC) \to \catC$ is a reflective localisation.
    \item \label{propsubitem:locally-cartesian}
    The reflective localisation $\colimit_{\catI} \from \Fun(\catI, \catC) \to \catC$ is locally cartesian if and only if $\catI$-shaped colimits are universal in $\catC$.
  \end{enumerate}
  \end{enumerate}
\end{proposition}

\begin{proof}
  Write $\Tw(\catI)$ for the twisted arrow $\infty$-category of $\catI$ as in \cite[\kerodontag{03JG}, \kerodontag{03JR}]{kerodon} and $\lambda \from \Tw(\catI) \to \catI^{\op} \times \catI$ for the natural projection \kerodon{03JK}.
  Given two diagrams $F, G \from \catI \to \catC$, the mapping anima $\Map(F, G)$ is computed as the limit of the diagram
  \[
    \begin{tikzcd}
      \Tw(\catI) \arrow[r, "{\lambda}"] & \catI^{\op} \times \catI \arrow[r, "{F^{\op} \times G}"] & \catC^{\op} \times \catC \arrow[r, "{\Map_{\catC}}"] & \Ani
    \end{tikzcd}
  \]
  by \cite[Prop. 5.1]{gepner2020}.
  Since the projection $\Tw(\catI) \to \catI^{\op} \times \catI \to \catI$ is a weak homotopy equivalence by \kerodon{048L} and $\catI$ is assumed to be weakly contractible, we conclude by applying \cite[\kerodontag{02XU}, \kerodontag{02N5}]{kerodon} to $\Tw(\catI) \to \ast$.
  This shows~\ref{propitem:constant-functor-ff}.
  Unwinding the definitions, we see that~\ref{propitem:colimit-reflective-localisation} is an immediate consequence of~\ref{propitem:constant-functor-ff}, since $\colimit_{\catI}$ is left adjoint to $\underline{(\blank)}$.
\end{proof}

Following \HTT{}{6.1.2.2}, for an infinity category $\catC$, we call the category $\Fun(\simplex^{\op}, \catC)$ the infinity category of \emph{simplicial objects} of $\catC$ and denote it by $s\catC$.

  \begin{corollary}
    \label{cor:geometric-realisation-is-locally-cartesian-in-profinite-anima}
    The geometric realisation functor
    \[
        \real = \colimit_{\simplex^{\!\op}} \from \simpl{\pfAni} = \Fun(\simplex^{\op}, \pfAni) \to \pfAni
    \]
    is a locally cartesian reflective localisation.
  \end{corollary}

  \begin{proof}
    The $\infty$-category $\pfAni$ admits all colimits and, by \SAG{}{E.6.3.2}, geometric realisations (i.e., $\simplex^{\op}$-shaped colimits) are universal.
  \end{proof}

\begin{proposition}
    \label{prop:geometric-realisation-of-pullbacks}
    Let $\catC$ be an $\infty$-category with fibre products and universal geometric realisations, $s \in \catC$ an object and $r \from s_{\bullet} \to \underline{s}$ in $\simpl{\catC}$ a simplicial resolution of $s$.
    Let $f \from x \to s$ be any map in $\catC$ and consider the fibre product
    \[
        \begin{tikzcd}
            x_{\bullet} \arrow[r, "{f_{\! \bullet}}"] \arrow[d, "{\pr_{\underline{x}}}"']
                \arrow[dr, phantom, very near start, "{ \lrcorner }"]
              & s_{\bullet} \arrow[d] \\
            \underline{x} \arrow[r, "{\underline{f}}"']
              & \underline{s}
        \end{tikzcd}
    \]
    in $\simpl{\catC}$.
    Then the following holds.
    \begin{thmlist}
      \item \label{propitem:resolution} The map $x_{\bullet} \to \underline{x}$ is a simplicial resolution of $x$.
      \item \label{propitem:homotopic} The geometric realisation $\real[f_{\! \bullet}]$ of $f_{\!\bullet}$ is homotopic to $f$.
      \item \label{propitem:fully-faithful}
      The following composite map is fully faithful:
      \[
      \begin{tikzcd}
        \overcat{\catC}{s} \arrow[r, "{\underline{(\blank)}}"] & \overcat{\simpl{\catC}}{\underline{s}} \arrow[r, "{\blank \times_{\underline{s}} s_{\bullet}}"] & \overcat{\simpl{\catC}}{s_{\bullet}}.
      \end{tikzcd}
      \]
    \end{thmlist}
  \end{proposition}

\begin{proof}
(1) By \Cref{cor:geometric-realisation-is-locally-cartesian-in-profinite-anima}, the geometric realisation functor $\real \from \simpl{\catC} \to \catC$ is a locally cartesian reflective localisation. Therefore, as claimed, we have
      \[
          \real[x_{\bullet}]
          = \real[\underline{x} \times_{\underline{s}} s_{\bullet}]
          = x \times_{s} \real[s_{\bullet}]
          = x.
      \]
(2) Again by \Cref{cor:geometric-realisation-is-locally-cartesian-in-profinite-anima}, the canonical map $c \from \real[x_{\bullet}] = \real[\underline{x} \times_{\underline{s}} s_{\bullet}] \to x \times_{s} \real[s_{\bullet}]$ is an isomorphism.
      The canonical map in question is uniquely determined by the following commutative diagram
      \[
      \begin{tikzcd}
                      \real[\underline{x} \times_{\underline{s}} s_{\bullet}] \arrow[rd, "{\simeq}", "c"'] \arrow[rrd, bend left = 15, "{\real[f_{\!\bullet}]}"] \arrow[dd, "{\real[\pr_{\underline{x}}]}"'] & & \\
                      & x \times_{s} \real[s_{\bullet}] = x \arrow[r, "f"] \arrow[d, "{\id_{x}}"'] \arrow[dr, phantom, very near start, "{ \lrcorner }"] & s = \real[s_{\bullet}] \arrow[d, "{\id_{s}}"] \\
                      \real[\underline{x}] \arrow[r, "{\epsilon_{x}}"', "{\simeq}"] & x \arrow[r, "f"'] & s,
      \end{tikzcd}
      \]
      where $\epsilon_{x} \from \real[\underline{x}] \to x$ denotes the counit equivalence of the adjunction $\real \ladj \underline{(-)}$.

(3) Both functors under consideration have left adjoints, given by
        \[
            \begin{tikzcd}
              \overcat{\simpl{\catC}}{\underline{s}} \arrow[r, "{\real}"] & \overcat{\catC}{s}
            \end{tikzcd}
            \andeq
            \begin{tikzcd}
              \overcat{\simpl{\catC}}{s_{\bullet}} \arrow[r, "{r \circ \blank}"] & \overcat{\simpl{\catC}}{\underline{s}}.
            \end{tikzcd}
        \]
        Since adjoints compose, this means that $\real[-] \circ (r \circ \blank)$ is left adjoint to the functor in question.
        By \Cref{lem:right-adjoint-faithful}, the assertion is hence equivalent to the counit $\counit$ of this adjunction to be an equivalence.
        Unraveling the definitions, we see that said counit is given componentwise by
        \[
            \counit_{y \to s} \from \real[(\underline{y} \times_{\underline{s}} s_{\bullet} \to s_{\bullet} \to \underline{s})] \to (y \to s),
        \]
        which is an equivalence by the preceding two assertions.
        This shows full faithfulness of
        \[
          \begin{tikzcd}
            \overcat{\catC}{s} \arrow[r, "{\underline{(\blank)}}"] & \overcat{\simpl{\catC}}{\underline{s}} \arrow[r, "{\blank \times_{\underline{s}} s_{\bullet}}"] & \overcat{\simpl{\catC}}{s_{\bullet}} .
          \end{tikzcd}
          \qedhere
        \]
\end{proof}

%%% Local Variables:
%%% mode: LaTeX
%%% TeX-master: "../haupt"
%%% End: